\documentclass[reqno,11pt]{amsart}
\usepackage[numbers]{natbib}
\usepackage[french,english]{babel}
\usepackage[utf8]{inputenc}
\usepackage[T1]{fontenc}
\usepackage{lmodern}
\usepackage{amsmath}
\usepackage{xfrac}
\usepackage{amssymb}
\usepackage{comment}
\usepackage{bm}
\usepackage{mathrsfs}
\usepackage{amsthm}
\usepackage{enumerate}
\usepackage{dsfont}
\usepackage[normalem]{ulem}
\usepackage{todonotes}

\usepackage{array}
\usepackage{tabularx}
\usepackage{booktabs}
\usepackage{ragged2e}

\usepackage{eurosym} 
\usepackage{stmaryrd} 

\usepackage{tikz}
\usetikzlibrary{arrows.meta}
\usepackage{graphicx}
\usepackage{multirow} 
\usepackage{float}    
\usepackage{adjustbox}

\usepackage{xr-hyper}
\usepackage{xurl}
\usepackage{hyperref}
\hypersetup{
  colorlinks=true,
  breaklinks=true,
  urlcolor=blue,
  linkcolor=blue,
  citecolor=magenta,
  pdfstartview = FitH,
  bookmarksopen = true
}

\usepackage{listings}
\usepackage{framed}

\usepackage[margin=1in]{geometry}
\usepackage{setspace}
\usepackage{setspace} 

\newcommand{\E}{\mathbb{E}}
\newcommand{\PP}{\mathbb{P}}
\newcommand{\Gen}{\mathsf{Gen}}

\makeatletter
\newcommand{\myitem}[2][]{%
  \item[#1]%
  \phantomsection%
  \def\@currentlabel{#1}%
  \label{#2}%
}
\newcommand{\myassitem}[2][]{%
  \item[#1]%
  \phantomsection%
  \def\@currentlabel{#1}%
  \label{#2}%
}
\DeclareRobustCommand{\myref}[1]{%
  \hyperref[#1]{\ref*{#1}}%
}
\makeatother

\newcounter{proofstep}

\DeclareMathOperator{\sgn}{sgn}

\renewcommand{\arraystretch}{1.35}
\title{The Martingale Property in Volterra Volatility Models}

\numberwithin{equation}{section}

\theoremstyle{plain}
\newtheorem{prooff}{Proof}[section]

\addto\captionsfrench{}

\newtheorem{lemma}[prooff]{Lemma}
\newtheorem{cor2}[prooff]{Corollary}

\theoremstyle{definition}
\newtheorem{example}{Example}
\newtheorem{rk}[prooff]{Remark}

\theoremstyle{definition}
\newtheorem{notation}{Notation}
\newtheorem{assumption}{Assumption}

\theoremstyle{plain}
\newtheorem{thm2}{Theorem}

\allowdisplaybreaks

\author{
  Julien Guyon$^{1,2}$,
  Thibault Jeannin$^{1}$,
  Benjamin Jourdain$^{1}$}

\date{}

\usepackage{fancyhdr}
\begin{document}

\begin{abstract}
We study martingality and strict local martingality in a broad class
of stochastic and path-dependent Volterra volatility models and
explain the underlying mechanisms.
Our martingality criteria allow rough kernels and rely on
non-accessibility of an explosive direction or one-sided growth
bounds on the Girsanov drift.
For smooth kernels, we derive sufficient conditions for strict local
martingality from an explosion criterion allowing multiplicative noise,
coupled factors and general initial input curves.
Our results recover classical criteria and yield new results
for Volterra and path-dependent models.
\end{abstract}

\maketitle

\renewcommand{\thefootnote}{\arabic{footnote}}

\footnotetext[1]{CERMICS, CNRS, ENPC, Institut Polytechnique de Paris, Marne-la-Vallée, France. Emails:
\href{mailto:julien.guyon@enpc.fr}{julien.guyon@enpc.fr},
\href{mailto:thibault.jeannin@enpc.fr}{thibault.jeannin@enpc.fr},
\href{mailto:benjamin.jourdain@enpc.fr}{benjamin.jourdain@enpc.fr}.}

\footnotetext[2]{NYU Tandon School of Engineering, Department of Finance and Risk Engineering, One MetroTech
Center, Brooklyn, NY 11201, USA. Email:
\href{mailto:julien.guyon@nyu.edu}{julien.guyon@nyu.edu}.\\ This research was supported by
the chair {\it Risques Financiers}, Fondation du Risque, and by the BNP Paribas chair {\it Futures of Quantitative Finance}.}

\makeatletter
\let\@fnsymbol\@arabic 
\makeatother

\renewcommand{\thefootnote}{} 
\addtocounter{footnote}{-1} 
\renewcommand{\thefootnote}{\arabic{footnote}} 

\bigskip
\noindent\textbf{Keywords:}
path-dependent volatility; stochastic Volterra equations; stochastic volatility; martingale property.

\bigskip
\noindent\textbf{Acknowledgements:}
The authors thank Paul Invernizzi for helpful discussions.
OpenAI's ChatGPT was used to review the wording, and all mathematical proofs were developed and written by the authors.

\bigskip
\tableofcontents

\section{Introduction}

\subsection{Motivation}

In financial modeling, as stressed by \cite{SeppRakhmonov2023}, verifying the martingale property of the asset price is
often a first sanity check: when the discounted price is a strict local martingale, several well-known pathologies can appear, including violations of put--call parity and distortions in the asymptotic
behavior of implied volatility wings \cite{CoxHobson2005,HestonLoewensteinWillard2007,JacquierKellerRessel2018}.
To prove martingality, beyond Novikov's criterion, which only applies to rather simple cases, the standard approach is to use a change-of-measure argument that reduces the problem to the study of explosion for an auxiliary SDE, as in \cite[Chapter 3.7]{McKean1969}.
Such an approach is used by \cite{Sin1998}, \cite{LionsMusiela2007}, and \cite{Jourdain2004}, which study, at different levels of generality, the martingale property and moment integrability of ``lognormal volatility models'' of the form
\begin{align}
\forall t\in[0,+\infty),\qquad
\frac{dS_t}{S_t} &= e^{X_t}\,dW_t^\rho,\notag\\
X_t
&=
x_0 e^{-\lambda t}
+
\frac{\mu}{\lambda}\left(1-e^{-\lambda t}\right)
+
\alpha\int_0^t e^{-\lambda(t-s)}\,dW_s.
\label{eq:lognormal_sv_Z}
\end{align}
Here we take \(x_0,\mu\in\mathbb R\) and \(\lambda,\alpha>0\). We define the Brownian motion \(W^\rho\) by
\[
\forall t \geq 0,\qquad
W_t^\rho
:=
\rho W_t+\sqrt{1-\rho^2}\,W_t^\perp, \quad \text{with} \quad \rho \in [-1,1].
\]
Here \(W\) and \(W^\perp\) are two independent Brownian motions. Then
\(W^\rho\) is a Brownian motion whose instantaneous correlation with \(W\) is
equal to \(\rho\). The authors show that the martingale property holds only
when \(\rho\le0\).
Later, \cite{BernardCuiMcLeish2017} obtained more general martingality criteria for SDE-based models.
In non-Markovian frameworks, \cite{Gassiat2018} studies a class of fractional stochastic volatility models, including the rough Bergomi model, of the form
\begin{align}
\forall t\in[0,+\infty),\qquad
\frac{dS_t}{S_t}&=\sigma(t,X_t)\,dW_t^\rho,
\qquad
X_t=\int_0^t K(t-s)\,dW_s.
\label{eq:gassiat_rough_bergomi}
\end{align}
Under the coefficient and kernel assumptions in \cite{Gassiat2018}, \(\rho<0\) gives martingality. For \(\rho>0\), a sufficient growth condition on \(\sigma\), as in \cite[Theorem~1(2), Eq.~(1.3)]{Gassiat2018}, gives strict local martingality when \(K\) dominates the rough kernel defined by
\begin{align}\label{eq:gassiat_rough_bergomi_kernel}
\forall \tau\in(0,+\infty),\ \forall H\in(0,\frac{1}{2}),\qquad
K_H(\tau)&:=\frac{\tau^{H-\frac{1}{2}}}{\Gamma(H+\frac{1}{2})}.
\end{align}
Here \(\Gamma\) denotes the Euler gamma function.
Adding a linear drift to the volatility factor in \eqref{eq:gassiat_rough_bergomi} gives the extension
\begin{align}
\forall t\in[0,+\infty),\qquad
\frac{dS_t}{S_t}&=\sigma(t,X_t)\,dW_t^\rho,
\qquad
X_t
=
\beta(t)
+
\kappa\int_0^t K(t-s)X_s\,ds
+
\nu\int_0^t K(t-s)\,dW_s.
\label{eq:volterra_stein_stein_intro}
\end{align}
Here \(\beta:\mathbb R_+\to\mathbb R\) is continuous and deterministic, \(\kappa\in\mathbb R,\ \nu\in\mathbb R\setminus\{0\}\), and \(K\) is locally square-integrable. This framework includes the Volterra--Stein--Stein model studied in \cite{AbiJaberHainautMotte2025},
when \(\sigma(t,x)=x\) or \(\sigma(t,x)=|x|\), while the case \(\kappa=0\) recovers Volterra volatility models of Bergomi type. In particular, with an exponential specification of \(\sigma\), one recovers rough/Volterra Bergomi-type models, including the original Bergomi model and, more generally, exponential Ornstein--Uhlenbeck (exp-OU) models, whereas a polynomial specification gives, for instance, the Quintic model introduced in \cite{AbiJaberIllandLi2022Quintic}. For the exponential Bergomi specifications below, we take \(\kappa=0\),
\(\nu>0\), and \(\sigma(t,x)=\varphi(t)e^x\), where
\(\varphi:\mathbb R_+\to(0,\infty)\) is continuous.
For the Quintic specification, one uses
\(\sigma(t,x)=\varphi(t)(a_0+a_1x+a_3x^3+a_5x^5)\), with
\(a_0,a_1,a_3\ge0\) and \(a_5>0\), and an exponential kernel
in \eqref{eq:volterra_stein_stein_intro} with \(\kappa=0\) and \(\nu>0\).

A large literature on \emph{rough volatility} has developed in recent years to capture both the strong path-dependence of volatility and its strong irregularity in time. \cite{rough} show empirically that log-volatility behaves, across a wide range of time scales, like a fractional process with Hurst parameter \(H\) close to \(0.1\), far below the Brownian value \(H=\frac{1}{2}\). This provides an empirical motivation for modeling volatility through rough Volterra-type dynamics rather than through classical Markovian diffusions.

In the rough Heston model of El~Euch and Rosenbaum \cite{ElEuchRosenbaum2019}, the variance satisfies a Volterra equation
with a fractional kernel.
The model retains the square-root diffusion coefficient
of the classical Heston model:
\begin{align}
\forall t\in[0,+\infty),\qquad
\frac{dS_t}{S_t}&=\sqrt{X_t}\,dW_t^\rho,
\notag\\
X_t
&=
X_0
+
\lambda
\int_0^t K_H(t-s)\left(\theta-X_s\right)\,ds
+
\nu
\int_0^t K_H(t-s)\sqrt{X_s}\,dW_s,
\label{RH}
\end{align}
where \(H\in(0,\frac{1}{2})\) and $K_H$ is defined by \eqref{eq:gassiat_rough_bergomi_kernel}, {\(X_0\ge0\), and \(\lambda,\theta,\nu>0\).}
The martingality of the rough Heston asset price was proved in
\cite{AbiJaberLarssonPulido2019} for all correlations \(\rho\).

A related line of research concerns quadratic rough models. Gatheral, Jusselin and Rosenbaum \cite{GatheralJusselinRosenbaum2020} introduced such a model to address the joint calibration of SPX and VIX smiles, showing that rough volatility combined with an explicit price-feedback (Zumbach) effect can overcome limitations of more classical continuous-path models. Dandapani, Jusselin and Rosenbaum \cite{DandapaniJusselinRosenbaum2021} provided a microstructural foundation for this class of models through quadratic Hawkes processes, while Rosenbaum and Zhang \cite{RosenbaumZhang2022} developed an efficient calibration method based on multi-factor approximations and deep neural networks. Using a resolvent, as proposed in
\cite{BourgeyGatheral2025}, we incorporate a linear mean-reversion term into the kernel.
In the fractional case, the transformed kernel is a Mittag--Leffler kernel.
This gives the following parametrization. We use the convention in which the sign of the Brownian driver in the asset price equation is absorbed into the correlation parameter \(\rho\). We define the function \(V\) by
\[
\forall x\in\mathbb R,\qquad V(x):=a_0+a_1(x-a_2)^2,
\qquad a_0\ge0,\quad a_1>0,\quad a_2\in\mathbb R.
\]
We set
\begin{align}
\forall t\in[0,+\infty),\qquad
\frac{dS_t}{S_t}&=\sqrt{V(X_t)}\,dW_t^\rho,
\qquad
X_t
=
\beta(t)+\int_{-\infty}^t K(t-s)\sqrt{V(X_s)}\,dW_s
\label{QRH2}
\end{align}
where \(K\) denotes the possibly transformed Volterra kernel and \(\beta\) the
corresponding input curve, chosen to be constant in \cite{BourgeyGatheral2025}. Integrals over negative times use a two-sided Brownian motion and a
prescribed history of \(X\). By a two-sided Brownian motion
\(W=(W_t)_{t\in\mathbb R}\), we mean a continuous process with
\(W_0=0\) such that \((W_t)_{t\ge0}\) and \((W_{-t})_{t\ge0}\)
are independent standard Brownian motions.
These history integrals are assumed to be well defined and to yield
continuous initial curves. A
minimal extension, the Quadratic Volterra Heston+ model of
\cite{BourgeyNoblePeturssonRosenbaumSzymanski2026QRHPlus}, adds a nonnegative term to \(V\) for \(x<x_0\) by defining
\begin{align} \label{V+}
\forall x\in\mathbb R,\qquad \widehat V(x)&:=V(x)+b_0\bigl((x_0-x)^+\bigr)^p,
\end{align}
with \(b_0\ge0\), \(x_0\in\mathbb R\), and \(p>0\), where \(x^+:=x\vee0\) and \(x^-:=(-x)\vee0\).
The model is given by
\begin{align}
\forall t\in[0,+\infty),\qquad
\frac{dS_t}{S_t}&=\sqrt{\widehat V(X_t)}\,dW_t^\rho,
\qquad
X_t=\beta(t)+\int_{-\infty}^t K(t-s)\sqrt{\widehat V(X_s)}\,dW_s.
\label{QRHplus}
\end{align}

When \(b_0=0\), \eqref{QRHplus} reduces to
\eqref{QRH2}.
For \(b_0=0\), \(a_0>0\), and \(\rho=1\), strict local martingality in the
Markovian case is already noted
in~\cite[p.~2]{BondiForde2026}.
The extension proved here concerns the smooth Volterra kernels in
Assumption~\myref{IV.2} below.

\medskip

Recent studies in quantitative finance, such as \cite{Zumbach2009},
\cite{Zumbach2010}, \cite{ChicheporticheBouchaud2014}, \cite{Bouchaud2021},
\cite{GuyonLekeufack2023} and \cite{AndreBoumezoudJourdain2023}, have highlighted the significant
path-dependence of volatility, which is largely explained by past asset returns---rather than by the current
asset price alone. These studies show that the instantaneous volatility of financial markets is mostly
endogenous and depends only marginally on additional random factors. Motivated by an empirical
investigation showing that a large fraction of the variations of an equity volatility index, such as the VIX, can be explained by \emph{past daily returns} and \emph{past squared daily returns}, Guyon and
Lekeufack \cite{GuyonLekeufack2023} introduced the class of path-dependent volatility (PDV)
models, where the instantaneous volatility depends explicitly on two factors: a (signed) return factor
\(R_{1,t}\) and a variance factor \(R_{2,t}\), built from past returns and past squared returns
weighted by memory kernels, respectively. A convenient formulation is the autonomous SVE
system of \cite{AndresJourdain2025}. We consider an asset price \((S_t)_{t\ge 0}\) evolving for \(t\ge 0\) with
instantaneous volatility \(\sigma_t=X_t\) according to
\begin{align} \label{PDVGL}
\forall t\in[0,+\infty), \qquad \frac{dS_t}{S_t} &= X_t\,dW_t, \\
X_t &= \beta_0+\beta_1 R_{1,t}+\beta_2 \sqrt{R_{2,t}}, \notag\\
R_{1,t} &= \int_{-\Delta}^{t} K_1(t-s)X_s\,dW_s, \notag\\
R_{2,t} &= \int_{-\Delta}^{t} K_2(t-s)X_s^{2}\,ds, \notag
\end{align}
where \(\beta_0\ge0\), \(\Delta\in[0,+\infty]\), and the nonnegative convolution
kernels \(K_1,K_2:(0,\infty)\to\mathbb R_+\) weight past returns and past
squared returns, respectively.
Throughout the paper, the PDV parameters satisfy
\(\beta_1<0\) and \(\beta_2>0\).

The negative-time path of the two-sided Brownian motion \(W\), independent of
its positive-time increments, encodes the history of the model before time zero,
namely the returns observed before time \(0\). When \(K_1\) and \(K_2\) each have two exponential components, the model admits a continuous-time four-factor PDV (4FPDV) representation: \(R_1\) and \(R_2\) can each be represented by two factors with different decay rates.
Nutz and Riveros Valdevenito \cite{NutzRiveros2024} prove existence
and pathwise uniqueness of global strong solutions for the 4FPDV model. In a more general (non-Markovian) Volterra-kernel setting, \cite{AndresJourdain2025} gives existence and uniqueness results for the Guyon--Lekeufack PDV
dynamics with general kernels. Gazzani and Guyon \cite{gazzani-guyon} give a method for pricing and calibration in the 4FPDV model. Following \cite{GuyonLekeufack2023}, they add an ``upside'' term to generate sufficiently high volatility when \(R_{1,t}>0\) and improve the fit of implied volatility smiles for large strikes. (An extra term is only needed in continuous-time models. Guyon and Parent \cite{guyon2025discrete} show that the discrete-time Guyon--Lekeufack model with fat-tailed innovations fits market smiles well, even for large strikes.) More precisely, they consider
\begin{align}
\forall t\in[0,+\infty),\qquad X_t&={\beta_0+}F_1(R_{1,t})+\beta_2\sqrt{R_{2,t}}, \label{eq:4PDV-upside}\\
\forall y\in\mathbb R,\qquad F_1(y)&:=\beta_1 y+\beta_{1,2}y^2\,\mathbf{1}_{\{y>0\}}. \notag
\end{align}
Here we assume \(\beta_{1,2}>0\).
We call the models obtained by allowing in \eqref{eq:4PDV-upside}
a function \(F_1\) that is nondecreasing for sufficiently large arguments 
\emph{upside-enhanced Volterra--Guyon--Lekeufack models}.
A related extension is the polynomial PDV model, introduced in \cite{baschetti2026polynomial}. We define its radicand and instantaneous volatility by
\begin{align}
\forall t\in[0,+\infty),\qquad
X_t
&:=
\beta_0+\beta_1 R_{1,t}+\beta_1^{(2)}R_{1,t}^2+\beta_2 R_{2,t},
\qquad
\sigma_t:=\sqrt{(X_t)^+}.
\label{eq:poly-pdv}
\end{align}
Here we have \(\beta_1^{(2)}>0\), and we assume that
\begin{align}
\forall T\in(0,+\infty),\qquad
\PP_0\left(\forall t\in[0,T],\ X_t\ge0\right)&=1.
\label{eq:poly-pdv-positivity}
\end{align}
A sufficient condition for \eqref{eq:poly-pdv-positivity} is
\(\beta_0+\beta_1y+\beta_1^{(2)}y^2\ge0\) for every \(y\in\mathbb R\),
since \(\beta_2>0\) and \(R_{2,t}\ge0\).

The linear PDV model of \citet{Pedersen2026LinearPDV} uses a weighted average of past volatility as its second factor.

For all PDV models considered below, we assume \(\sigma_0>0\).
For the two-factor Guyon--Lekeufack model, \cite{NutzRiveros2024} prove that the
asset price is a martingale when volatility is bounded from below.
We extend this result to more general kernels. In the Volterra--Guyon--Lekeufack
model~\eqref{PDVGL}, under the conditional law \(\PP_0\) defined in
Section~\ref{subsec:unified-framework}, the following implication holds for every \(T>0\)
\[
\PP_0\left(\forall t\in[0,T],\ \sigma_t\ge0\right)=1
\quad\Longrightarrow\quad
(S_t)_{t\in[0,T]}\ \text{is a martingale under} \quad \mathbb{P}_0.
\]

\medskip

Recently, Abi Jaber, Gassiat and Sotnikov \cite{AbiJaberGassiatSotnikov2025}
proved sharp martingality criteria for finite-order signature volatility models.
Their results cover models of polynomial type and extend to signatures driven by Ornstein--Uhlenbeck processes, including the Quintic OU model. For Brownian signatures truncated at order \(N\ge2\)
and a nonzero coefficient of \(\frac{W_t^N}{N!}\) in the volatility, with \(\rho\ne0\),
the price is a martingale if and only if \(N\) is odd
and this coefficient has the opposite sign to \(\rho\).
\medskip
\medskip

Bondi and Pulido \cite{BondiPulido2024} give a Feller-type boundary
test for scalar stochastic Volterra equations with nonsingular kernels
and constant initial input.
We study nonlinear functions of coupled Volterra factors with initial
curves encoding the past.
For smooth kernels, we prove an explosion criterion with multiplicative
noise by controlling the memory drift in a semimartingale representation.

\medskip

Martingality follows either by excluding an explosive direction while
controlling the opposite direction, or from one-sided growth bounds
on the Girsanov drift.
Both criteria allow rough kernels.
The second adapts Lyapunov methods for moment control to Volterra
factors that need not be semimartingales, using a Volterra chain-rule
inequality when Itô's formula cannot be applied.
Related estimates were used by Bourdon and
Jeannin \cite{BourdonJeannin} for subcritical moments in the rough
Bergomi model.

 \medskip

The applications of our martingality criteria to several popular Volterra models in mathematical finance are summarized in Table~\ref{Tableapplication}. Below, we consider the following two generalized Volterra--Guyon--Lekeufack specifications, with \(\beta_0=\beta_{0,1}\) outside \(F_1\), in which the continuous function \(F_1\) is defined, for \(\beta_{0,1},\beta_{0,2},y_0\in\mathbb R\) and \(\beta_{1,1}<0<\beta_{1,2}\), by
\begin{align}
\forall y\in\mathbb R,\qquad
F_1(y)
&:=
{(\beta_{1,1}y)
\vee
(\beta_{0,2}-\beta_{0,1}+\beta_{1,2}y)},
\tag{A}\label{eq:F1A}
\\
\forall y\in\mathbb R,\qquad
F_1(y)
&:=
{(\beta_{1,1}y)
\vee
\left(\beta_{0,2}-\beta_{0,1}+\beta_{1,2}\sqrt{(y-y_0)^+}\right)}.
\tag{B}\label{eq:F1B}
\end{align}

\clearpage

\begin{table}[H]
\centering
\caption{Sufficient conditions for martingality and strict local
martingality in the listed models. Throughout the table, \(\rho\neq0\).
The martingality results allow rough kernels under the assumptions of the results cited in the last column,
whereas the strict local martingality results proved here require the smooth-kernel
regularity in Assumption~\myref{IV.2} below, but are  strongly expected to extend to rough kernels.
The same criteria hold with \(\PP\) replaced throughout by the
conditional law \(\PP_0\) defined in
Section~\ref{subsec:unified-framework}, for almost every initial history.}
\label{Tableapplication}

\scriptsize
\renewcommand{\arraystretch}{1.25}
\setlength{\tabcolsep}{2pt}

\begin{tabularx}{\linewidth}{
|>{\hsize=1.040000\hsize\linewidth=\hsize\RaggedRight\arraybackslash}X
|>{\hsize=0.560000\hsize\linewidth=\hsize\RaggedRight\arraybackslash}X
|>{\hsize=0.680000\hsize\linewidth=\hsize\RaggedRight\arraybackslash}X
|>{\hsize=1.720000\hsize\linewidth=\hsize\RaggedRight\arraybackslash}X|}
\hline
\textbf{Model}
&
\textbf{Martingality on $[0,T]$}
&
\textbf{Strict local martingality on $[0,T]$}
&
\textbf{Result and reference}
\\
\hline

Volterra--Stein--Stein \eqref{eq:volterra_stein_stein_intro}
&
always
&
never
&
Theorem~\ref{maintheorem}\myref{(A)} and~\myref{(Bii)} below.
Here \(\sigma(t,x)=x\) or \(\sigma(t,x)=|x|\). 
For \(\sigma(t,x)=x\), this recovers a consequence of
\cite[Theorem~3.3]{AbiJaber2022Gaussian}.
\\
\hline

Volterra--Heston \eqref{RH}
&
always
&
never
&
Theorem~\ref{maintheorem}\myref{(A)} and~\myref{(Bii)} below.
Our results also give this known result.
\\
\hline

Lognormal stochastic volatility \eqref{eq:lognormal_sv_Z}
&
$\rho<0$
&
$\rho>0$
&
Theorem~\ref{maintheorem}\myref{(Bii)} below;
Theorem~\ref{nonmartingality} and Lemma~\ref{lem:accessibility_models} below. Our results also give these known results.
\\
\hline

Quintic model of \cite{AbiJaberIllandLi2022Quintic}
&
$\rho<0$
&
$\rho>0$
&
Theorem~\ref{maintheorem}\myref{(Bii)} below;
Theorem~\ref{nonmartingality} and Lemma~\ref{lem:accessibility_models} below. Here \(a_5>0\). Our results also give these known results.
\\
\hline

Volterra--Bergomi \eqref{eq:gassiat_rough_bergomi}
&
\(\rho<0\), \(\sigma\ge0\), and
\(\sigma(t,x)\le C_T(1+|x|)\) for \(x\le0\), \(t\le T\)
&
\(\rho>0\) and, for some \(\chi>1\),
\(\sigma(t,x)\ge c_Tx^\chi\) for all sufficiently large \(x\)
and \(t\le T\), with \(c_T>0\)
&
Theorem~\ref{maintheorem}\myref{(A)} and~\myref{(Bii)} below;
Corollary~\ref{corollary} and Lemma~\ref{lem:accessibility_models} below.
Our results recover criteria in \cite{Gassiat2018}
for smooth kernels. For strict local martingality, \cite{Gassiat2018}
also allows rough kernels by exploiting the additive Gaussian
structure.
\\
\hline

Quadratic Volterra Heston+ \eqref{QRHplus} (\(1\le p\le2\))
&
-- 
&
Under Assumption~\myref{IV.2} below: always
&
Corollary~\ref{corollary} and Lemma~\ref{lem:accessibility_models} below, with \(a_0>0\), \(a_1>0\), and \(1\le p\le2\). New results.
\\
\hline

Volterra--Guyon--Lekeufack \eqref{PDVGL}
&
\(\PP(\forall t\in[0,T],\ \sigma_t\ge0)=1\)
&
\(
\begin{gathered}
\exists t_0\in[0,T):\\[-1mm]
\PP(\sigma_{t_0}<0)>0
\end{gathered}
\)
&
Corollary~\ref{corollary} below. New result extending \cite{NutzRiveros2024},
which proves martingality when volatility is almost surely bounded below.
\\
\hline

Linear PDV
&
\(\PP(\forall t\in[0,T],\ \sigma_t\ge0)=1\)
&
\(
\begin{gathered}
\exists t_0\in[0,T):\\[-1mm]
\PP(\sigma_{t_0}<0)>0
\end{gathered}
\)
&
Corollary~\ref{corollary} below. See also \cite[Theorem~5]{Pedersen2026LinearPDV} for martingality.
\\
\hline

Upside-enhanced Volterra--Guyon--Lekeufack, specification \eqref{eq:F1A}
&
never
&
always
&
Theorem~\ref{nonmartingality} and Lemma~\ref{lem:accessibility_models} below. New result.
\\
\hline

Upside-enhanced Volterra--Guyon--Lekeufack, specification \eqref{eq:F1B}
&
always
&
never
&
Theorem~\ref{maintheorem}\myref{(A)} and~\myref{(Bii)} below. New result.
\\
\hline

Polynomial PDV \eqref{eq:poly-pdv}
&
never
&
always
&
Theorem~\ref{nonmartingality} and Lemma~\ref{lem:accessibility_models} below. New model.
\\
\hline
\end{tabularx}
\end{table}

\section{Framework and main results}

\subsection{A unified Volterra framework}\label{subsec:unified-framework}

\noindent On a filtered probability space $(\Omega,\mathcal F,(\mathcal F_t)_{t\ge0},\PP)$ satisfying the usual conditions, we consider a Brownian motion
$(W_t)_{t\ge0}$. 
We assume that a regular conditional
probability given \(\mathcal F_0\) exists. Unless stated otherwise, we work under the regular conditional law defined by
\(\PP_0:=\PP(\,\cdot\mid\mathcal F_0)\), and write \(\E_0\) for its expectation.
We fix a realization of the \(\mathcal F_0\)-history outside an
exceptional \(\PP\)-null set, so that \(\PP_0\) is a fixed probability
measure. All assertions under this law are understood for almost every
initial history. For an integrable random variable \(Z\),
\(\E_0[Z\mid\mathcal F_t]\) denotes its conditional expectation
under this measure, given the full history up to time \(t\).
We consider two initial curves
\(G_1,G_2:\mathbb R_+\to\mathbb R\), an initial price \(S_0>0\), and a continuous adapted factor process
\(Y:=(Y_1,Y_2)\) solving \eqref{PDVeq}.
The curves \(G_1,G_2\), viewed as random elements of
\(C(\mathbb R_+,\mathbb R)\), and \(S_0\) are
\(\mathcal F_0\)-measurable.
For statements under \(\PP\), we additionally assume \(\E[S_0]<\infty\).
We enlarge the space, if necessary, by a Brownian motion $(W_t^\perp)_{t\ge0}$ independent of the factor system \((W,Y)\), and take, for \(t\ge0\), \(\mathcal F_t\) to be the usual augmentation of
\[
\mathcal F_0\vee
{\Gen\left((W_s)_{s\in[0,t]},(Y_s)_{s\in[0,t]},(W_s^\perp)_{s\in[0,t]}\right)}.
\]
Here \(\Gen(\cdot)\) denotes the generated sigma-field.
Under \(\PP_0\), \(W\) and \(W^\perp\) are independent Brownian motions.
We consider a discounted asset price \((S_t)_{t\ge0}\). For \(\rho\in[-1,1]\), we set \(W^\rho=\rho W+\sqrt{1-\rho^2}W^\perp\). The price satisfies
\begin{align*}
\forall t\in[0,+\infty), \qquad \frac{dS_t}{S_t} &= \sigma_t\,dW_t^\rho,
\qquad \sigma_t:=\sigma(t,X_t),
\end{align*}
with
\begin{align}
\forall t\in[0,+\infty), \qquad X_t &= \beta(t)+F_1(Y_{1,t})+F_2(Y_{2,t}), \notag\\
Y_{1,t} &=
G_1(t)+\int_0^t K_1(t-s)\Gamma_1(s,X_s)\,dW_s, \tag{$\mathrm{PDV}$}\label{PDVeq} \\
Y_{2,t} &= G_2(t)+\int_0^t K_2(t-s)\,\Gamma_2(s,X_s)\,ds, \notag
\end{align}
for deterministic functions
\[
{F_1,F_2:\mathbb R\to\mathbb R,} \qquad 
{\Gamma_1,\Gamma_2:\mathbb R_+\times\mathbb R\to\mathbb R,}
\qquad
\sigma:\mathbb R_+\times \mathbb{R} \to\mathbb R,
\qquad
\beta:\mathbb R_+\to\mathbb R.
\]
Throughout the paper, the kernels are convolution kernels:
{\(K_1,K_2:(0,\infty)\to\mathbb R_+\) are measurable, and we use
\(K_i(t-s)\) for \(s<t\).} 
We call \(Y_1\) and \(Y_2\) the first and second Volterra factors.
For the purposes of this paper, we extend the coefficient \(x\mapsto\sqrt{x}\), initially defined on \(\mathbb R_+\), continuously to \(\mathbb R\) by \(x\mapsto\sqrt{x^+}\).
\begin{example}
This framework includes most of the models considered in the literature.
The representation of each model is given in Table~\ref{tab:models_general_framework}. For the first five models, we take \(G_1=G_2=0\). 
In the Quadratic Volterra Heston+ row, \(G_1\) also includes the input curve
\(\beta\) from \eqref{QRHplus}, so the framework coefficient \(\beta\) is zero.
For linear PDV, \(G_1\) includes the contribution of the constant risk-free rate.
The first two rows use \(K_1=K_2=K\).
\begin{table}[H]
\centering
\caption{Examples of models covered by the general framework \eqref{PDVeq}.}
\label{tab:models_general_framework}

\scriptsize
\renewcommand{\arraystretch}{1.25}
\setlength{\tabcolsep}{2pt}

\begin{tabularx}{\linewidth}{
|>{\hsize=2.300000\hsize\linewidth=\hsize\RaggedRight\arraybackslash}X
|>{\hsize=0.242000\hsize\linewidth=\hsize\RaggedRight\arraybackslash}X
|>{\hsize=0.440000\hsize\linewidth=\hsize\RaggedRight\arraybackslash}X
|>{\hsize=0.740000\hsize\linewidth=\hsize\RaggedRight\arraybackslash}X
|>{\hsize=1.540000\hsize\linewidth=\hsize\RaggedRight\arraybackslash}X
|>{\hsize=1.185000\hsize\linewidth=\hsize\RaggedRight\arraybackslash}X
|>{\hsize=1.045000\hsize\linewidth=\hsize\RaggedRight\arraybackslash}X
|>{\hsize=0.990000\hsize\linewidth=\hsize\RaggedRight\arraybackslash}X
|>{\hsize=1.045000\hsize\linewidth=\hsize\RaggedRight\arraybackslash}X
|>{\hsize=0.473000\hsize\linewidth=\hsize\RaggedRight\arraybackslash}X|}
\hline
\textbf{Model}
& \(\boldsymbol{\rho}\)
& \(\boldsymbol{s}\)
& \(\boldsymbol{\varepsilon_s}\)
& \(\boldsymbol{\sigma(t,x)}\)
& \(\boldsymbol{F_1(y)}\)
& \(\boldsymbol{\Gamma_1(t,x)}\)
& \(\boldsymbol{F_2(r)}\)
& \(\boldsymbol{\Gamma_2(t,x)}\)
& \(\boldsymbol{\beta(t)}\) \\
\hline

Volterra--Stein--Stein \eqref{eq:volterra_stein_stein_intro}
& $\rho$
&
\(\pm1\)
&
\(s\sgn(\nu)\)
& \(\sigma(t,x)\)
& \(\nu y\)
& \(1\)
& $\kappa r$
& $x$
& $\beta(t)$
\\
\hline

Volterra--Heston \eqref{RH}
&
\(\rho\)
&
\(\pm1\)
&
\(s\)
&
\(\sqrt{x^+}\)
&
\(\nu y\)
&
\(\sqrt{x^+}\)
&
\(-\lambda r\)
&
\(x\)
&
\(\beta(t)\)
\\
\hline

Lognormal stochastic volatility \eqref{eq:lognormal_sv_Z}
&
\(\rho\)
&
\(\pm1\)
&
\(s\)
&
\(e^x\)
&
\(y\)
&
\(\alpha\)
&
\(0\)
&
\(0\)
&
\(\beta(t)\)
\\
\hline

Quintic model of \cite{AbiJaberIllandLi2022Quintic}
&
\(\rho\)
&
\(\pm1\)
&
\(s\)
&
\(\begin{gathered}\varphi(t)(a_0+a_1x\\{}+a_3x^3\\{}+a_5x^5)\end{gathered}\)
&
\(y\)
&
\(\nu\)
&
\(0\)
&
\(0\)
&
\(\beta(t)\)
\\
\hline

Volterra--Bergomi \eqref{eq:gassiat_rough_bergomi}
&
\(\rho\)
&
\(\pm1\)
&
\(s\)
&
\(\sigma(t,x)\)
&
\(y\)
&
\(\nu\)
&
\(0\)
&
\(0\)
&
\(\beta(t)\)
\\
\hline

Quadratic Volterra Heston+ \eqref{QRHplus} (\(1\le p\le2\))
& $\rho$
&
\(\pm1\)
&
\(s\)
& $\sqrt{\widehat V(x)}$
& $y$
& $\sqrt{\widehat V(x)}$
& $0$
& $0$
& $0$
\\
\hline

Volterra--Guyon--Lekeufack \eqref{PDVGL}
& $1$
&
\(\pm1\)
&
\(-s\)
& $x$
& $\beta_1 y$
& $x$
& $\beta_2\sqrt{r^+}$
& $x^2$
& $\beta_0$
\\
\hline

Linear PDV \cite{Pedersen2026LinearPDV}
& $1$
& \(\pm1\)
& \(-s\)
& $x$
& $\beta_1 y$
& $x$
& $\beta_2 r$
& $x$
& $\beta_0$
\\
\hline

Upside-enhanced Volterra--Guyon--Lekeufack, specification \eqref{eq:F1A}
&
\(1\)
&
\(\pm1\)
&
\(+1\)
&
\(x\)
&
\eqref{eq:F1A}
&
\(x\)
&
\(\beta_2\sqrt{r^+}\)
&
\(x^2\)
&
\(\beta_{0,1}\)
\\
\hline

Upside-enhanced Volterra--Guyon--Lekeufack, specification \eqref{eq:F1B}
&
\(1\)
&
\(\pm1\)
&
\(+1\)
&
\(x\)
&
\eqref{eq:F1B}
&
\(x\)
&
\(\beta_2\sqrt{r^+}\)
&
\(x^2\)
&
\(\beta_{0,1}\)
\\
\hline

Polynomial PDV \eqref{eq:poly-pdv}
&
\(1\)
&
\(\pm1\)
&
\(+1\)
&
\(\sqrt{x^+}\)
&
\(\begin{gathered}\beta_1y\\{}+\beta_1^{(2)}y^2\end{gathered}\)
&
\(\sqrt{x^+}\)
&
\(\beta_2r\)
&
\(x^+\)
&
\(\beta_0\)
\\
\hline
\end{tabularx}
\end{table}
\end{example}
\begin{rk}
This representation is not unique: different choices of the functions \(F_i\), \(\Gamma_i\), and \(\sigma\) may correspond to the same model.
\end{rk}
We use the following assumptions only where stated. Kernel regularity is summarized in Table~\ref{tab:kernel-assumptions},
and polynomial-growth exponents in Table~\ref{homog}.

\begin{assumption}\label{assumption1}
The convolution kernels \(K_1,K_2\) are nonnegative and satisfy
\[
 \forall t\in(0,+\infty),\quad\forall i\in\{1,2\},\qquad
 \int_0^tK_i(s)\,ds>0.
\]

\begin{enumerate}

\myassitem[1.1]{I.3}
\textbf{Increment regularity of the kernels.}
There exist \(\delta_1,\delta_2>0\) such that, for every \(T>0\),
the following bounds hold with a constant \(C_T>0\):
\begin{align}
\forall h\in[0,T],\qquad
\left(
\int_0^{T-h}
|K_1(u+h)-K_1(u)|^2\,du
+
\int_0^h K_1(u)^2\,du
\right)^{\frac{1}{2}}
&\le
C_T h^{\delta_1},
\label{eq:K1_increment_regular_assumption}\\
\int_0^{T-h}
|K_2(u+h)-K_2(u)|\,du
+
\int_0^h K_2(u)\,du
&\le
C_T h^{\delta_2}.
\label{eq:K2_increment_regular_assumption}
\end{align}

\myassitem[1.2]{I.5}
\textbf{Polynomial growth of $F_i$ and $\Gamma_i$.} The functions \(G_i\), \(F_i\), \(\Gamma_i\), \(\sigma\), and \(\beta\) are continuous,
and there exist exponents \(\gamma_i\ge 0\), \(i\in\{1,2\}\),
\(f_{1,+},f_{1,-},f_2>0\), and a constant
\(C_{F_2}>0\), such that, with
$f_1:=f_{1,+}\vee f_{1,-}$,
\(\gamma_i f_i\le1\), \(i\in\{1,2\}\), and for every
\(T>0\), there exists \(C_T>0\) satisfying
\begin{align*}
\forall i\in\{1,2\},\quad
\forall (t,x)\in {[0,T]}\times\mathbb R,\qquad
|\Gamma_i(t,x)|
&\le
C_T(1+|x|^{\gamma_i}),\\
\forall y\in\mathbb R,\qquad
|F_2(y)|
&\le
C_{F_2}(1+|y|^{f_2}),
\end{align*}
and, for each $s\in\{-1,1\}$, there exists $\varepsilon_s\in\{-1,1\}$
such that, for some \(y_*>0\), $F_1(sy)$ has sign $\varepsilon_s$ for $y\ge y_*$ and
\[
\exists\,0<c\le C<+\infty
\text{ such that }
\forall y\in[y_*,+\infty),\qquad
cy^{f_{1,s}}\le \varepsilon_s F_1(sy)\le Cy^{f_{1,s}}.
\]
\end{enumerate}
\end{assumption}

\medskip

\begin{table}[H]
\centering
\caption{Polynomial-growth exponents in Assumption~\myref{I.5} for functions defined in Table~\ref{tab:models_general_framework}.}
\label{homog}

\scriptsize
\renewcommand{\arraystretch}{1.25}
\setlength{\tabcolsep}{2pt}

\begin{tabularx}{\linewidth}{
|>{\hsize=2.380000\hsize\linewidth=\hsize\RaggedRight\arraybackslash}X
|>{\hsize=0.700000\hsize\linewidth=\hsize\RaggedRight\arraybackslash}X
|>{\hsize=0.700000\hsize\linewidth=\hsize\RaggedRight\arraybackslash}X
|>{\hsize=0.840000\hsize\linewidth=\hsize\RaggedRight\arraybackslash}X
|>{\hsize=0.700000\hsize\linewidth=\hsize\RaggedRight\arraybackslash}X
|>{\hsize=0.840000\hsize\linewidth=\hsize\RaggedRight\arraybackslash}X
|>{\hsize=0.840000\hsize\linewidth=\hsize\RaggedRight\arraybackslash}X|}
\hline
\textbf{Model} & \(\boldsymbol{\gamma_1}\) & \(\boldsymbol{f_1}\) & \(\boldsymbol{\gamma_1 f_1}\) & \(\boldsymbol{\gamma_2}\) & \(\boldsymbol{f_2}\) & \(\boldsymbol{\gamma_2 f_2}\) \\
\hline

Volterra--Stein--Stein \eqref{eq:volterra_stein_stein_intro}
& \(0\) & \(1\) & \(0\) & \(1\) & \(1\) & \(1\) \\
\hline

Volterra--Heston \eqref{RH}
& \(\frac{1}{2}\) & \(1\) & \(\frac{1}{2}\) & \(1\) & \(1\) & \(1\) \\
\hline

Lognormal stochastic volatility \eqref{eq:lognormal_sv_Z}
&
\(0\)
&
\(1\)
&
\(0\)
&
\(0\)
&
\(1\)
&
\(0\)
\\
\hline

Quintic model of \cite{AbiJaberIllandLi2022Quintic}
&
\(0\)
&
\(1\)
&
\(0\)
&
\(0\)
&
\(1\)
&
\(0\)
\\
\hline

Volterra--Bergomi \eqref{eq:gassiat_rough_bergomi}
&
\(0\)
&
\(1\)
&
\(0\)
&
\(0\)
&
\(1\)
&
\(0\)
\\
\hline

Quadratic Volterra Heston+ \eqref{QRHplus} (\(1\le p\le2\))
& \(1\) & \(1\) & \(1\) & \(0\) & \(f_2>0\) & \(0\) \\
\hline

Volterra--Guyon--Lekeufack \eqref{PDVGL}
& \(1\) & \(1\) & \(1\) & \(2\) & \(\frac{1}{2}\) & \(1\) \\
\hline

Linear PDV \cite{Pedersen2026LinearPDV}
& \(1\) & \(1\) & \(1\) & \(1\) & \(1\) & \(1\) \\
\hline

Upside-enhanced Volterra--Guyon--Lekeufack, specification \eqref{eq:F1A}
& \(1\) & \(1\) & \(1\) & \(2\) & \(\frac{1}{2}\) & \(1\) \\
\hline

Upside-enhanced Volterra--Guyon--Lekeufack, specification \eqref{eq:F1B}
& \(1\) & \(1\) & \(1\) & \(2\) & \(\frac{1}{2}\) & \(1\) \\
\hline

Polynomial PDV \eqref{eq:poly-pdv}
& \(\frac{1}{2}\) & \(2\) & \(1\) & \(1\) & \(1\) & \(1\) \\
\hline
\end{tabularx}
\par\smallskip
\noindent The quadratic upside specification does not satisfy
Assumption~\myref{I.5}, since \(\gamma_1f_1=2>1\).
The two specifications \eqref{eq:F1A} and \eqref{eq:F1B} do satisfy its growth bounds.
\end{table}
For a bounded stopping time \(\theta\), we consider
\begin{align}
\forall t\in[0,+\infty),\qquad
\widetilde X_t^\theta
&=
\beta(t)+F_1(\widetilde Y_{1,t}^\theta)
+F_2(\widetilde Y_{2,t}^\theta), \notag\\
\widetilde Y_{1,t}^\theta
&={G_1(t)+}
\int_0^t K_1(t-s)
\Gamma_1(s,\widetilde X_s^\theta)\,dW_s \notag\\
&\quad
+
\rho
\int_0^t
\mathbf{1}_{\{s\ge\theta\}}
K_1(t-s)
\Gamma_1(s,\widetilde X_s^\theta)
\sigma(s,\widetilde X_s^\theta)\,ds,
\tag{\(\widetilde{\mathrm{PDV}}_\theta\)}\label{tildePDV}\\
\widetilde Y_{2,t}^\theta
&={G_2(t)+}
\int_0^t K_2(t-s)
\Gamma_2(s,\widetilde X_s^\theta)\,ds. \notag
\end{align}
We set \(\widetilde Y^\theta:=(\widetilde Y_1^\theta,\widetilde Y_2^\theta)\).
In particular, for all $t\in [0,\theta]$, we have $(\widetilde X_t^\theta,\widetilde Y_t^\theta)=(X_t,Y_t).$
By Lemma~\ref{martingality0} below, martingality of \(S\) after \(\theta\)
is equivalent to non-explosion of \(\widetilde X^\theta\).
For \(T,n>0\), we define the probability measure \(\mathbb Q_n\) by
\[
\left.\frac{d\mathbb Q_n}{d\PP_0}\right|_{\mathcal F_{\theta+T}}
:=\frac{S_{\theta+(T\wedge\tau_{(\theta,n)}^X)}}{S_\theta},
\]
where
\(\tau_{(\theta,n)}^X:=\inf\{u\ge0:\ |X_{\theta+u}|\ge n\}\).
(For all hitting times, we use the convention \(\inf\varnothing=+\infty\).)
Up to \(\theta+(T\wedge\tau_{(\theta,n)}^X)\), this change of measure adds the following drift to the first Volterra factor:
\begin{align}
&\rho\int_0^t\mathbf1_{\{s\ge\theta\}}K_1(t-s)
\Gamma_1(s,\widetilde X_s^\theta)
\sigma(s,\widetilde X_s^\theta)\,ds.
\label{girsanovdrift}
\end{align}
The main results use the well-posedness assumptions in
Assumption~\ref{standingassumption} below.


\begin{assumption}\label{standingassumption}
We require the systems \eqref{PDVeq} and \eqref{tildePDV}
to be well posed in the following sense.
\begin{enumerate}

\myassitem[2.1]{standing_i}
The system \eqref{PDVeq} has a global continuous solution \((X,Y)\), that is unique in law.

\myassitem[2.2]{standing_ii}
For every bounded stopping time \(\theta\), \eqref{tildePDV}
has a continuous solution
\((\widetilde X^\theta,\widetilde Y^\theta)\) up to explosion, that is unique in law. In particular, for every \(n>0\), with
$\tau_{(\theta,n)}^{\widetilde X^\theta}
:=
\inf\{u\ge0:\
|\widetilde X_{\theta+u}^{\theta}|\ge n
\},$ the conditional law given \(\mathcal F_\theta\) of
$(
\widetilde X_{t\wedge(\theta+\tau_{(\theta,n)}^{\widetilde X^\theta})}^{\theta},
\widetilde Y_{t\wedge(\theta+\tau_{(\theta,n)}^{\widetilde X^\theta})}^{\theta}
)_{t\ge0}$
is the same for all such solutions.
\end{enumerate}
\end{assumption}
We evaluate a solution of \eqref{tildePDV} only before its
explosion time
$\theta+
\lim\limits_{n\to\infty}\tau_{(\theta,n)}^{\widetilde X^\theta}.$ A stopped integrand is defined to be zero after its
stopping time, without evaluating the solution there.

\begin{rk}
The weak well-posedness requirement is stated in Assumption~\ref{standingassumption}.
Under Assumption~\ref{assumption1} and for the continuous initial curves
\(G_1,G_2\) in the framework, a sufficient condition for well-posedness
of \eqref{PDVeq} is that \(F_1,F_2\) are locally Lipschitz and
\(\Gamma_1,\Gamma_2\) are locally Lipschitz in space, locally uniformly
in time.
For \eqref{tildePDV}, it is sufficient to assume in addition that the
product \(\Gamma_1\sigma\), which appears in the additional drift,
is locally Lipschitz in space, locally uniformly in time.
The kernel increment bounds in Assumption~\myref{I.3} and polynomial growth bounds in Assumption~\myref{I.5} then
give global well-posedness of \eqref{PDVeq} and local well-posedness
of \eqref{tildePDV} up to explosion, including the conditional
uniqueness in law in Assumption~\myref{standing_ii}.

\medskip

This criterion applies to Volterra--Stein--Stein with \(\sigma(t,x)=x\)
or \(\sigma(t,x)=|x|\), lognormal stochastic volatility, the Quintic model,
and Volterra--Bergomi when \(\sigma\) is locally Lipschitz in space,
locally uniformly in time.
It also covers Quadratic Volterra Heston, since \(\sqrt V\) is smooth
for $a_0>0$.

Weak well-posedness can also hold when the coefficients are not locally
Lipschitz on the whole state space. The following models can be treated
by localization or by weak existence and uniqueness results.

\begin{itemize}
\item \textbf{PDV models.}
For the following models, we have \(K_2,\Gamma_2\ge0\), so assuming
\(G_2(t)>0\) for every \(t\ge0\) gives \(Y_{2,t}\ge G_2(t)>0\).
All lower bounds stated here also hold for the corresponding
transformed processes \eqref{tildePDV} before explosion.
When \(F_2(w)=\beta_2\sqrt{w^+}\), this prevents its singularity
at zero from being reached.

\begin{itemize}
\item \emph{Guyon--Lekeufack and specification~\eqref{eq:F1A}.}
Here \(F_1\) is locally Lipschitz.
For Guyon--Lekeufack, the above localization is used in
the literature \cite{NutzRiveros2024,AndresJourdain2025} under the respective
kernel and initial-data conditions.

\item \emph{Specification~\eqref{eq:F1B}.}
If \(\beta_{0,1}+\beta_{1,1}y_0>\beta_{0,2}\), the affine branch
dominates near \(y_0\), where \(F_1(y)=\beta_{1,1}y\).
Since both branches are locally Lipschitz away from \(y_0\),
\(F_1\) is locally Lipschitz on \(\mathbb R\).

\item \emph{Polynomial PDV.}
If \(\beta_0+F_1(y)\ge0\) for every \(y\in\mathbb R\), then we have
\[
\forall t\ge0,\qquad
X_t=\beta_0+F_1(Y_{1,t})+\beta_2Y_{2,t}
\ge\beta_2G_2(t)>0.
\]
Thus the singularity of \(\Gamma_1(x)=\sqrt{x^+}\) at zero
is also avoided.
\end{itemize}
In each case, the stated conditions recover the locally Lipschitz
setting by localization.

\item \textbf{Rough and Volterra--Heston models.}
Here \(\Gamma_1(x)=\sqrt{x^+}\) is not locally Lipschitz at zero.
Existence and uniqueness in law are provided by the affine Volterra
results of Abi Jaber, Larsson and Pulido \cite{AbiJaberLarssonPulido2019},
under their kernel and initial-data conditions.
\end{itemize}
\end{rk}

\subsection{Martingality}

\begin{notation}\label{not:sign-conventions}
The sign \(s\in\{-1,1\}\) specifies whether large positive or
large negative values of \(Y_1\) are considered.
The sign \(\varepsilon_s\) is the sign of \(F_1(sy)\) for sufficiently
large \(y\), as specified in Assumption~\myref{I.5}.
The sign \(s\) is chosen separately in each assumption or result where it
is used. When \(\rho\neq0\), \(\sgn(\rho)\) denotes the sign of \(\rho\).
\end{notation}

We first explain the sign conditions in
Assumptions~\ref{assumption_subcritical_bad_branch} and~\myref{V.1} below. For \(\rho\neq0\) and \(s\in\{-1,1\}\), we have
\begin{align} 
s\,\widetilde Y_{1,t}^{\theta}
&=
sG_1(t)
+
s
\int_{0}^{t}
K_1(t-u)
\Gamma_1(u,\widetilde X_u^{\theta})\,dW_u \notag
\\
&\qquad
+
s\rho
\int_{0}^{t}
\mathbf{1}_{\{u\ge\theta\}}
K_1(t-u)
\Gamma_1(u,\widetilde X_u^{\theta})
\sigma(u,\widetilde X_u^{\theta})\,du.
\label{opeho}
\end{align}
To fix ideas, we first discuss the case where
\(s\rho\Gamma_1\sigma\ge0\).
\begin{itemize}
\item {\textbf{Case 1: Potentially destabilizing drift.}}
When \(s\rho\Gamma_1\sigma\ge0\), the Girsanov drift tends to increase \(s\widetilde Y_1^\theta\) and push it towards large positive values. By Assumption~\myref{I.5}, there exists \(z^*>0\) such that
\[
\forall z\ge z^*,\qquad
\varepsilon_s F_1(sz)\ge0.
\]
If \(s\,\widetilde Y_{1,t}^{\theta}\ge z^*\), then
$F_1(\widetilde Y_{1,t}^\theta)
=
F_1(s
(s\,\widetilde Y_{1,t}^{\theta}))$
satisfies \(\varepsilon_s F_1(\widetilde Y_{1,t}^\theta)\ge0\), and
\begin{align} \label{varespilonx}
\varepsilon_s \widetilde X_t^\theta
&=
\varepsilon_s \beta(t)
+
|
F_1(\widetilde Y_{1,t}^\theta)|
+
\varepsilon_s F_2(\widetilde Y_{2,t}^\theta).
\end{align}
Heuristically, growth of \(s\widetilde Y_1^\theta\) therefore
tends to increase \(\varepsilon_s\widetilde X^\theta\) and push it towards large positive values
through \eqref{varespilonx}, once the contribution of \(F_2\)
is controlled. If \(s\rho\Gamma_1\sigma\) grows sufficiently
in the direction, \(\varepsilon_s x\to+\infty\), in the sense of \eqref{eq:explosive_drift_growth},
the additional drift in \eqref{opeho}
can in turn push \(s\widetilde Y_1^\theta\) further upward,
in the direction \(s\widetilde Y_1^\theta\to+\infty\).
The explosion mechanism is summarized in Figure~\ref{fig:feedbackloop}.
The exponent \(\chi\) is defined in \eqref{eq:explosive_drift_growth}
of Assumption~\myref{V.1} below.
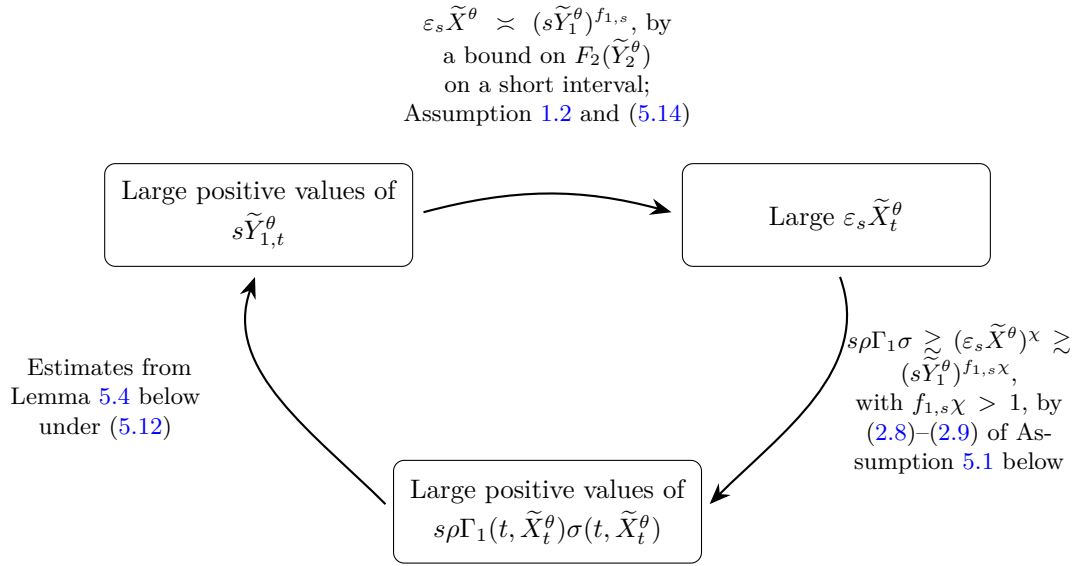
\begin{figure}[H]
\centering
\resizebox{0.95\textwidth}{!}{%
\begin{tikzpicture}[
    box/.style={
        draw,
        rounded corners,
        align=center,
        text width=3.7cm,
        minimum height=1.35cm,
        inner sep=5pt,
        font=\small
    },
    arrow/.style={
        -{Stealth[length=2.8mm,width=2.2mm]},
        thick,
        shorten >=4pt,
        shorten <=4pt
    },
    lab/.style={
        font=\footnotesize,
        align=center,
        inner sep=2pt
    }
]

\node[box] (Z) at (-3.8,1.5) {
Large positive values of\\[2pt]
\(s\widetilde Y_{1,t}^{\theta}\)
};

\node[box] (X) at (3.8,1.5) {
Large \(\varepsilon_s \widetilde X_t^\theta\)
};

\node[box] (G) at (0,-2.4) {
Large positive values of\\[2pt]
\(s\rho\Gamma_1(t, \widetilde X_t^\theta)
\sigma(t,\widetilde X_t^\theta)\)
};

\draw[arrow]
(Z.east) to[out=15,in=165]
node[lab, above, yshift=8mm, text width=5.2cm] {
\(\varepsilon_s\widetilde X^\theta
\asymp(s\widetilde Y_1^\theta)^{f_{1,s}}\), by\\
a bound on \(F_2(\widetilde Y_2^\theta)\)\\
on a short interval;\\
Assumption~\myref{I.5} and \(\eqref{eq:X_bounds_dyadic_linear}\)
}
(X.west);


\draw[arrow]
(X.south) to[out=-70,in=45]
node[lab, right, text width=3.7cm] {
\(s\rho\Gamma_1\sigma\gtrsim
(\varepsilon_s\widetilde X^\theta)^\chi
\gtrsim(s\widetilde Y_1^\theta)^{f_{1,s}\chi}\),\\
with \(f_{1,s}\chi>1\), by
\(\eqref{eq:explosive_drift_growth}\)--\(\eqref{eq:structural_superlinear_threshold}\) of Assumption~\myref{V.1} below
}
(G.east);
\draw[arrow]
(G.west) to[out=135,in=-110]
node[lab, left, text width=4.5cm] {
{ Estimates from\\
Lemma~\ref{lem:dyadic_step_Y_linear} below\\
under \eqref{eq:dyadic_strip_bounds_linear}}
}
(Z.south);

\end{tikzpicture}
}
\caption{Positive feedback loop leading to explosion of \eqref{tildePDV}.}
\label{fig:feedbackloop}
\end{figure}
For a sufficiently short time interval \([\theta,t]\), the growth bounds on \(F_2,\Gamma_2\) in Assumption~\myref{I.5}, the condition
\(\gamma_2f_2\le1\), and the regularity of \(K_2\) in Assumption~\myref{I.3} give, before explosion,
\[
\sup_{u\in[\theta,t]}|F_2(\widetilde Y_{2,u})|
\le B+\frac{1}{4}\sup_{u\in[0,t]}|\widetilde X_u|,
\]
as shown in \eqref{eq:Y2_absorption_bound} in the proof of explosion below. 
The following bounds hold when \(s\widetilde Y_1\) remains bounded
from below up to explosion.
The contribution of \(F_1(\widetilde Y_1)\) when
\(s\widetilde Y_1<0\) is then bounded and can be absorbed into \(B\).
The identity
\[
\varepsilon_s\widetilde X_u
=\varepsilon_s\beta(u)
+\varepsilon_sF_1(\widetilde Y_{1,u})
+\varepsilon_sF_2(\widetilde Y_{2,u})
\]
and the polynomial bounds on \(F_1\) in Assumption~\myref{I.5} give
\begin{align*}
c\left(\sup_{u\in[\theta,t]}(s\widetilde Y_{1,u})^+\right)^{f_{1,s}}-B
&\le \sup_{u\in[\theta,t]}\varepsilon_s\widetilde X_u\le B+C\left(\sup_{u\in[\theta,t]}(s\widetilde Y_{1,u})^+\right)^{f_{1,s}}.
\end{align*}
Thus,
\(s\widetilde Y_1\) explodes towards \(+\infty\) if and only if
\(\varepsilon_s\widetilde X\) explodes towards \(+\infty\),
and their explosion times coincide.

\medskip

\item {\textbf{Case 2: Stabilizing drift.}}
When $s\rho\Gamma_1\sigma\ge0$, the Girsanov drift stabilizes
the negative side: if $s\widetilde Y_1^\theta$ is negative, the Girsanov drift pushes $s\widetilde Y_1^\theta$ towards
zero and cannot drive it towards $-\infty$.
\end{itemize}
If \(s\rho\Gamma_1\sigma\le0\), the same discussion applies
with \(s\) replaced by \(-s\): the roles of the positive
and negative sides are interchanged.
\begin{rk}
Except for the Quintic and Stein--Stein models, all models studied here
satisfy
\[
\forall(t,x)\in\mathbb R_+\times\mathbb R,\qquad
\Gamma_1(t,x)\sigma(t,x)\ge0.
\]
The choice \(s=\sgn(\rho)\) therefore gives \(s\rho\Gamma_1\sigma\ge0\).
Where \(\Gamma_1\sigma\le0\), choosing \(s=-\sgn(\rho)\)
gives \(s\rho\Gamma_1\sigma\ge0\).
For the Quintic model, \(\Gamma_1\sigma\) has the sign of \(x\)
on both tails. When \(\rho>0\), both choices
\((s,\varepsilon_s)=(1,1)\) and \((-1,-1)\) therefore give
$s\rho\Gamma_1(t,x)\sigma(t,x)>0$ for sufficiently large $sx$.
\end{rk}

\begin{assumption}\label{assumption_subcritical_bad_branch}
Suppose that Assumption~\ref{assumption1}
holds. We use the polynomial growth exponents from Assumption~\myref{I.5}.

\begin{enumerate}
\myassitem[3.1]{III.1}
\textbf{Non-accessibility of the explosive region.}
We assume \(\rho\ne0\). There exists \(s\in\{-1,1\}\) such that,
for every \(T>0\), there are constants \(L_T\in\mathbb R\) and \(C_T>0\)
such that
\[
\sup_{t\in[0,T]}\varepsilon_s X_t\le L_T
\qquad\PP_0\text{-a.s.},
\]
and
\begin{align} \label{xboud}
\forall(t,x)\in[0,T]\times\mathbb R,\qquad
  \varepsilon_s x\le L_T\ \Longrightarrow\
  -s\rho\Gamma_1(t,x)\sigma(t,x)
  \le C_T\left(1+|x|^{\frac{1}{f_{1,-s}}}\right).
 \end{align}
\myassitem[3.2]{III.2}
\textbf{One-sided growth bound for the Girsanov drift.}
For every \(T>0\), there exist \(R_T,C_T\ge1\) such that
\[
\begin{gathered}
\forall s\in\{-1,1\},\quad\forall(t,x)\in[0,T]\times\mathbb R, \qquad
\varepsilon_sx\ge R_T\ \Longrightarrow\
s\rho\Gamma_1(t,x)\sigma(t,x)
\le C_T\left(1+|x|^{\frac{1}{f_{1,s}}}\right).
\end{gathered}
\]
\end{enumerate}
\end{assumption}
Sufficient conditions for martingality are given in Assumptions~\myref{III.1} and~\myref{III.2}. In Assumption~\myref{III.2}, other Lyapunov functions could be used to control the solutions of \eqref{tildePDV}. More precisely:
\begin{itemize}
\item
The upper bound on $\varepsilon_sX$ is given in Assumption~\myref{III.1}.
By the stopped change-of-measure argument, for every \(T>0\),
the bound \(\varepsilon_s\widetilde X_t^0\le L_T\) holds almost surely
for all \(t\in[0,T]\) before explosion. Thus, Assumption~\myref{III.1} rules out explosion along
$\varepsilon_s\widetilde X_t^0\to+\infty$, and hence along
$s\widetilde Y_{1,t}^0\to+\infty$.
It remains to control the opposite direction
$-s\widetilde Y_{1,t}^0\to+\infty$ through the one-sided growth
bound \eqref{xboud} in Assumption~\myref{III.1} with exponent $\frac{1}{f_{1,-s}}$.
If $s\rho\Gamma_1\sigma\ge0$, \eqref{xboud} is trivially satisfied.

\item On each tail selected by \(\varepsilon_s\),
Assumption~\myref{III.2} controls \(s\rho\Gamma_1(t,x)\sigma(t,x)\) from above when \(\varepsilon_s x\) is large, namely \(\varepsilon_s x\ge R_T\).
 \item The bound \eqref{xboud} holds with \(s=\sgn(\rho)\) if
 \[
  \forall(t,x)\in\mathbb R_+\times\mathbb R,\qquad
  \Gamma_1(t,x)\sigma(t,x)\ge0.
 \]

\end{itemize}

The one-sided drift bound in Assumption~\myref{III.2} is checked for the listed
models in Table~\ref{tab:verification-III3}.
\begin{table}[H]
\centering
\caption{Verification of Assumption~\myref{III.2}.}
\label{tab:verification-III3}

\scriptsize
\renewcommand{\arraystretch}{1.25}
\setlength{\tabcolsep}{2pt}

\begin{tabularx}{\linewidth}{
|>{\hsize=1.360000\hsize\linewidth=\hsize\RaggedRight\arraybackslash}X
|>{\hsize=1.120000\hsize\linewidth=\hsize\RaggedRight\arraybackslash}X
|>{\hsize=0.640000\hsize\linewidth=\hsize\RaggedRight\arraybackslash}X
|>{\hsize=0.880000\hsize\linewidth=\hsize\RaggedRight\arraybackslash}X|}
\hline
\textbf{Model}
&
\(\boldsymbol{\rho\Gamma_1(t,x)\sigma(t,x)}\)
&
\(\boldsymbol{(f_{1,+},f_{1,-})}\)
&
\textbf{Assumption}\par\textbf{\myref{III.2}}
\\
\hline

Volterra--Stein--Stein \eqref{eq:volterra_stein_stein_intro}
&
\(\rho \sigma(t,x)\)
&
\((1,1)\)
&
Yes for \(\sigma(t,x)=x\) or \(\sigma(t,x)=|x|\)
\\
\hline

Volterra--Heston \eqref{RH}
&
\(\rho 
x^+
\)
&
\((1,1)\)
&
Yes
\\
\hline

Lognormal stochastic volatility \eqref{eq:lognormal_sv_Z}
&
\(\rho \alpha e^x\)
&
\((1,1)\)
&
Yes if \(\rho<0\); no if \(\rho>0\)
\\
\hline

Quintic model of \cite{AbiJaberIllandLi2022Quintic}
&
\(\begin{gathered}\rho \nu\varphi(t)(a_0+a_1x +a_3x^3+a_5x^5)\end{gathered}\)
&
\((1,1)\)
&
Yes if \(\rho<0\); no if \(\rho>0\)
\\
\hline

Volterra--Bergomi \eqref{eq:gassiat_rough_bergomi}
&
\(\rho \nu\sigma(t,x)\)
&
\((1,1)\)
&
For \(\sigma\ge0\): yes if \(\sigma(t,x)\le C_T(1+|x|)\) on \(\rho x\ge0\), \(t\le T\)
\\
\hline

Quadratic Volterra Heston+ \eqref{QRHplus} (\(1\le p\le2\))
&
\(\rho 
\widehat V(x)
\)
&

\((1,1)\)

&
No if \(\rho\neq0\), for \(1\le p\le2\)
\\
\hline

Volterra--Guyon--Lekeufack \eqref{PDVGL}
&
\(\rho 
x^2
\)
&
\((1,1)\)
&
No
\\
\hline

Linear PDV \cite{Pedersen2026LinearPDV}
&
\(x^2\)
&
\((1,1)\)
&
No
\\
\hline

Upside-enhanced Volterra--Guyon--Lekeufack, specification \eqref{eq:F1A}
&
\(\rho 
x^2
\)
&
\((1,1)\)
&
No
\\
\hline

Upside-enhanced Volterra--Guyon--Lekeufack, specification \eqref{eq:F1B}
&
\(\rho 
x^2
\)
&
\(\left(\frac{1}{2},1\right)\)
&
Yes
\\
\hline

Polynomial PDV \eqref{eq:poly-pdv}
&
\(\rho 
x^+
\)
&
\((2,2)\)
&
No
\\
\hline
\end{tabularx}
\end{table}

\begin{notation}
For a locally finite signed Borel measure \(\mu\) on \(\mathbb R_+\)
and a measurable function \(h\), we define
\[
\forall t\in[0,+\infty),\qquad
(h*\mu)(t)=(\mu*h)(t):=\int_{[0,t]}h(t-s)\,\mu(ds),
\]
whenever the integral is well defined. If \(h\) and \(g\) are two
locally integrable functions on \(\mathbb R_+\), we write
\[
(h*g)(t):=\int_0^t h(t-s)g(s)\,ds,
\qquad \text{for a.e. }t>0.
\]

A locally finite signed Borel measure \(\mathcal R_1\) on
\(\mathbb R_+\) is called a \emph{resolvent of the first kind}
of a locally integrable kernel \(K_1\) if
\[
(K_1*\mathcal R_1)(t)=(\mathcal R_1*K_1)(t)=1,
\qquad \text{for a.e. }t>0.
\]
\end{notation}

\begin{assumption}\label{assumption2}
\textbf{Kernel conditions.}
We use the following two conditions separately.
\begin{enumerate}

\myassitem[4.1]{IV.1}
\textbf{Nonnegative resolvent for the martingality criterion in Theorem~\ref{maintheorem}\myref{(Bii)} below.}
The kernel \(K_1\) belongs to \(L^1_{\mathrm{loc}}(\mathbb R_+)\), is nonnegative, and has a nonnegative resolvent of the first kind
of the form
\[
\mathcal R_1=a_{\rm res}\delta_0+\ell_{\rm res}(t)\,dt,
\]
where \(\delta_0\) is the unit mass at zero, \(a_{\rm res}\ge0\), and
\(\ell_{\rm res}\in L^1_{\mathrm{loc}}(\mathbb R_+)\) is nonnegative and
nonincreasing.

\myassitem[4.2]{IV.2}
\textbf{Smooth-kernel regularity for the explosion argument.}
The kernel \(K_1\) has a continuous extension to \([0,\infty)\) satisfying
\[
K_1(0)>0,
\qquad
K_1\in C^1([0,+\infty)),
\qquad
K_1'\in BV_{\mathrm{loc}}([0,+\infty)).
\]
Here \(BV_{\mathrm{loc}}\) denotes functions of locally bounded variation.
\end{enumerate}
\end{assumption}

\begin{rk}
The main classes of kernels and possible choices of the regularity
exponents are listed in Table~\ref{tab:kernel-assumptions}.
We define the two-parameter Mittag--Leffler function, for \(\alpha,\beta>0\), by
\(E_{\alpha,\beta}(z):=\sum_{k=0}^{\infty}\frac{z^k}{\Gamma(\alpha k+\beta)}\).
\begin{table}[H]
\centering
\caption{Regularity exponents and resolvent properties of standard kernels.}
\label{tab:kernel-assumptions}

\scriptsize
\renewcommand{\arraystretch}{1.35}
\setlength{\tabcolsep}{2pt}

\begin{tabularx}{\linewidth}{
|>{\RaggedRight\arraybackslash}p{0.23\linewidth}
|>{\RaggedRight\arraybackslash}p{0.22\linewidth}
|>{\Centering\arraybackslash}p{0.13\linewidth}
|>{\Centering\arraybackslash}p{0.13\linewidth}
|>{\Centering\arraybackslash}p{0.12\linewidth}
|>{\Centering\arraybackslash}X|}
\hline

\textbf{Kernel}
&
\textbf{Resolvent of the first kind}
&
\(\boldsymbol{\delta_1}\)
\textbf{in Assumption}\par\textbf{\myref{I.3}}
&
\(\boldsymbol{\delta_2}\)
\textbf{in Assumption}\par\textbf{\myref{I.3}}
&
\textbf{Assumption}\par\textbf{\myref{IV.1}}
&
\textbf{Assumption}\par\textbf{\myref{IV.2}}
\\
\hline

\textbf{Finite positive linear combination of exponentials}
\par\smallskip
\(\displaystyle
K(t)=\sum_{i=1}^{N}c_i e^{-\lambda_i t}
\)
\par
\(N\in\mathbb N^*,\ c_i>0,\ \lambda_i\ge0\)
&
\par\smallskip
\(\displaystyle
\mathcal R
=
\left(\sum_i c_i\right)^{-1}\delta_0+\ell_{\rm exp}(t)\,dt,
\)
\par
where \(\ell_{\rm exp}\) is nonnegative and nonincreasing~\cite[Chapter~5, Theorem~5.4]{GripenbergLondenStaffans1990}.
&
\(\displaystyle\frac{1}{2}\)
&
\(1\)
&
Yes
&
Yes
\\
\hline

\textbf{Fractional rough kernel}
\par\smallskip
\(\displaystyle
K_H(t)
=
\frac{t^{H-\frac{1}{2}}}
{\Gamma(H+\frac{1}{2})}
\)
\par
\(H\in(0,\frac{1}{2})\)
&
\(\displaystyle
\mathcal R_H(dt)
=
\frac{t^{-H-\frac{1}{2}}}
{\Gamma(\frac{1}{2}-H)}\,dt
\)
\par\smallskip
\cite[Table~1]{AbiJaberLarssonPulido2019}
&
\(H\)
&
\(\displaystyle H+\frac{1}{2}\)
&
Yes
&
No:
singular at \(0\)
\\
\hline

\textbf{Mittag--Leffler kernel}
\par
\textbf{(rough exponential)}
\par\smallskip
\(\displaystyle
K_{H,\lambda}(t)
=c\,t^{H-\frac{1}{2}}
\)
\par
\(\displaystyle
{}\times E_{H+\frac{1}{2},H+\frac{1}{2}}
\left(-\lambda t^{H+\frac{1}{2}}\right)
\)
\par
\(c>0,\ \lambda\ge0,\ H\in(0,\frac{1}{2})\)
&
\(\displaystyle
\mathcal R_{H,\lambda}(dt)
=
\frac{1}{c}
\left(
\frac{t^{-H-\frac{1}{2}}}
{\Gamma(\frac{1}{2}-H)}
+\lambda
\right)dt
\)
\par\smallskip
\cite[Table~1]{AbiJaberLarssonPulido2019}
&
\(H\)
&
\(\displaystyle H+\frac{1}{2}\)
&
Yes
&
No:
singular at \(0\)
\\
\hline
\end{tabularx}
\end{table}
\end{rk}

\begin{thm2}\label{maintheorem}
Suppose that Assumption~\ref{standingassumption} holds.
Then the following statements hold.
\begin{itemize}

\myitem[(A)]{(A)}
If \(\rho=0\), then \((S_t)_{t\ge0}\) is a martingale.

\myitem[(B)]{(B)}
Suppose in addition that \(\rho\neq0\) and that
Assumption~\ref{assumption1} holds. The following statements hold.
\begin{itemize}

\myitem[(B.i)]{(Bi)}
\textbf{Martingality without entering the explosive region.}
If Assumption~\myref{III.1} holds, then
\((S_t)_{t\ge0}\) is a martingale.

\myitem[(B.ii)]{(Bii)}
\textbf{Martingality under a one-sided drift bound.}
If Assumptions~\myref{III.2} and~\myref{IV.1} hold, then
\((S_t)_{t\ge0}\) is a martingale.

\end{itemize}
\end{itemize}
\end{thm2}

\subsection{Strict local martingality}

\begin{assumption}\label{assumption_explosive_branch}
\textbf{Explosive branch and accessibility.}
Suppose that Assumption~\ref{assumption1} holds. We use the notation
for the polynomial-growth exponents in Assumption~\myref{I.5}.
We assume that \(\rho\neq0\).
\begin{enumerate}
\myassitem[5.1]{V.1}
\textbf{Growth and monotonicity conditions for explosion.}
There exist
\(s\in\{-1,1\}\), \(z_0\in\mathbb R\),
\(x^*\ge0\), \(\chi>0\), such that
\begin{align}
z&\longmapsto {\varepsilon_s F_1(sz)}
\quad\text{is nondecreasing on }[z_0,+\infty).
\label{eq:power_branch_F1}
\end{align}
For every \(T>0\), there exists \(c_T>0\) such that
\begin{align}
\forall (t,x)\in[0,T]\times\mathbb R,\qquad
\varepsilon_s x\ge x^*
\ \Longrightarrow\
s\rho\Gamma_1(t,x)\sigma(t,x)
\ge c_T(\varepsilon_s x)^\chi.
\label{eq:explosive_drift_growth}
\end{align}
We also assume
\begin{align}
f_{1,s}\chi&>1.
\label{eq:structural_superlinear_threshold}
\end{align}

\myassitem[5.2]{V.2}
\textbf{Accessibility of the explosive region.}
With \(s,\varepsilon_s,x^*,z_0\) as in Assumption~\myref{V.1}, we define
the stopping time
\[
\tau_*
:=\inf\{u\ge0:\varepsilon_s X_u>x^*,\ sY_{1,u}>z_0\},
\]

and assume that
\[
\PP_0(\tau_*<\infty)>0.
\]
By continuity of \(X,Y_1\), this condition is equivalent to
\[
\exists t_0\ge0,\qquad
\PP_0\left(\varepsilon_s X_{t_0}>x^*,\ sY_{1,t_0}>z_0\right)>0.
\]
\end{enumerate}
\end{assumption}

We use \(+\infty+u=+\infty\) for \(u\ge0\).
We define the sigma-field \(\mathcal F_\infty\) by
\(\mathcal F_\infty:=\Gen(\bigcup_{t\ge0}\mathcal F_t)\).
For a possibly infinite stopping time \(\tau\), we define the sigma-field
\(\mathcal F_\tau\) by
\[
\mathcal F_\tau
:=\{A\in\mathcal F_\infty:
\forall t\in[0,+\infty),\quad A\cap\{\tau\le t\}\in\mathcal F_t\}.
\]

\begin{thm2}\label{nonmartingality}
Suppose that Assumptions~\ref{assumption1},~\ref{standingassumption}
and~\myref{IV.2} hold,
and that Assumption~\myref{V.1} holds.

\begin{itemize}
\myitem[(i)]{nonmartingality_i}
For every \(t_0\ge0\) and \(t>t_0\), if
\[
\PP_0\left(\varepsilon_s X_{t_0}>x^*,\
sY_{1,t_0}>z_0\right)>0,
\]
then we have \(\E_0[S_t]<\E_0[S_{t_0}]\), and
\((S_u)_{u\in[t_0,t]}\) is a strict local martingale
with respect to \((\mathcal F_u)_{u\in[t_0,t]}\).

\myitem[(ii)]{nonmartingality_ii}
If Assumption~\myref{V.2} holds, then, for every \(h>0\), we have
\[
\E_0\left[
\mathbf1_{\{\tau_*<\infty\}}S_{\tau_*+h}
\,\middle|\,\mathcal F_{\tau_*}
\right]
<S_{\tau_*}
\]
a.s.\ on \(\{\tau_*<\infty\}\).
For \(u\ge0\), the variable
\(\mathbf1_{\{\tau_*<\infty\}}S_{\tau_*+u}\) is defined to be zero on
\(\{\tau_*=\infty\}\).
\end{itemize}

\end{thm2}

\begin{rk} \label{1234}
The parameters used in the applications are listed in Table~\ref{tab:verification_assumption_III5}.
In this table, we choose \(s=\sgn(\rho)\).
The accessibility properties used for strict local martingality are stated in Lemma~\ref{lem:accessibility_models} below.
The Volterra--Bergomi row also assumes the growth condition stated
in Table~\ref{Tableapplication}.

\begin{table}[H]
\centering
\caption{Verification of Assumption~\myref{V.1}. Here \(\tau_*\) is defined in Assumption~\myref{V.2}.}
\label{tab:verification_assumption_III5}

\scriptsize
\renewcommand{\arraystretch}{1.25}
\setlength{\tabcolsep}{2pt}

\begin{tabularx}{\linewidth}{
|>{\hsize=1.880000\hsize\linewidth=\hsize\RaggedRight\arraybackslash}X
|>{\hsize=0.600000\hsize\linewidth=\hsize\RaggedRight\arraybackslash}X
|>{\hsize=0.640000\hsize\linewidth=\hsize\RaggedRight\arraybackslash}X
|>{\hsize=1.120000\hsize\linewidth=\hsize\RaggedRight\arraybackslash}X
|>{\hsize=0.800000\hsize\linewidth=\hsize\RaggedRight\arraybackslash}X
|>{\hsize=0.840000\hsize\linewidth=\hsize\RaggedRight\arraybackslash}X
|>{\hsize=1.040000\hsize\linewidth=\hsize\RaggedRight\arraybackslash}X
|>{\hsize=1.080000\hsize\linewidth=\hsize\RaggedRight\arraybackslash}X|}
\hline
\textbf{Model}
& \(\boldsymbol{\varepsilon_s}\)
& \(\boldsymbol{x^*}\)
& \(\boldsymbol{z_0}\)
& \(\boldsymbol{(\chi,f_{1,s})}\)
& \(\boldsymbol{f_{1,s}\chi>1}\)
& \textbf{Assumption}\par\textbf{\myref{V.1} applied}
& \(\boldsymbol{\PP_0(\tau_*\le t_0)>0}\)\newline
  \(\boldsymbol{\forall t_0>0}\) \\
\hline

Volterra--Stein--Stein \eqref{eq:volterra_stein_stein_intro}
&
\(\sgn(\rho\nu)\)
&
arbitrary
&
arbitrary
&
\((1,1)\)
&
false
&
No
&
Not needed.
\\
\hline

Volterra--Heston \eqref{RH}
& \(\sgn(\rho)\)
& arbitrary
& arbitrary
& \((1,1)\)
& false
& No
& Not needed. \\
\hline

Lognormal stochastic volatility \eqref{eq:lognormal_sv_Z}
&
\(\sgn(\rho)\)
&
arbitrary
&
arbitrary
&
\((\chi,1)\)\par\(\chi>1\)
&
true
&
Yes if \(\rho>0\); no if \(\rho<0\).
&
Yes.
\\
\hline

Quintic model of \cite{AbiJaberIllandLi2022Quintic}
&
\(\sgn(\rho)\)
&
arbitrary
&
arbitrary
&
\((5,1)\)
&
true
&
Yes if \(\rho>0\); no if \(\rho<0\).
&
Yes.
\\
\hline

Volterra--Bergomi \eqref{eq:gassiat_rough_bergomi}
&
\(\sgn(\rho)\)
&
large enough for \eqref{eq:explosive_drift_growth}
&
arbitrary
&
\((\chi,1)\)
&
true
&
Yes if \(\rho>0\).
&
Yes.
\\
\hline

Quadratic Volterra Heston+ \eqref{QRHplus} (\(1\le p\le2\))
& \(\sgn(\rho)\)
& large enough; arbitrary if \(a_0>0\)
& arbitrary
& \((2,1)\)
& true
& Yes
& Yes if \(a_0>0\). \\
\hline

Volterra--Guyon--Lekeufack \eqref{PDVGL}
& \(-1\)
& \(0\)
& \(0\)
& \((2,1)\)
& true
& Yes
& Depends on the kernels and initial curves. \\
\hline

Linear PDV \cite{Pedersen2026LinearPDV}
& \(-1\)
& \(0\)
& \(0\)
& \((2,1)\)
& true
& Yes
& Depends on the kernels and initial curves. \\
\hline

Upside-enhanced Volterra--Guyon--Lekeufack, specification \eqref{eq:F1A}
& \(+1\)
& \(0\)
& \(\frac{\beta_{0,1}-\beta_{0,2}}{\beta_{1,2}-\beta_{1,1}}\)
& \((2,1)\)
& true
& Yes
& Yes. \\
\hline

Upside-enhanced Volterra--Guyon--Lekeufack, specification \eqref{eq:F1B}
&
\(+1\)
&
arbitrary
&
\(\begin{gathered}y_0\vee\\\frac{\beta_{0,2}-\beta_{0,1}}{\beta_{1,1}}\end{gathered}\)
&
\(\left(2,\frac{1}{2}\right)\)
&
false
&
No
&
Not established here; not needed for martingality.
\\
\hline

Polynomial PDV \eqref{eq:poly-pdv}
& \(+1\)
& \(0\)
& \(-\frac{\beta_1}{2\beta_1^{(2)}}\)
& \((1,2)\)
& true
& Yes
& Yes. \\
\hline
\end{tabularx}
\end{table}
\end{rk}

The following corollary applies to most models considered here,
for which \(F_1\) is linear, including the Volterra--Heston
\cite{ElEuchRosenbaum2019}, Bergomi \cite{Gassiat2018},
Quadratic Volterra Heston+
\cite{BourgeyNoblePeturssonRosenbaumSzymanski2026QRHPlus}, and
Guyon--Lekeufack \cite{GuyonLekeufack2023} models.

\begin{cor2}[Linear $F_1$] \label{corollary}
Suppose that Assumptions~\ref{assumption1}
and~\ref{standingassumption} hold,
that \(\rho\neq0\), that \(F_1(y)=Ay\)
for some \(A\in\mathbb R^\ast\), and that
\(\Gamma_1(t,x)\sigma(t,x)\ge0\) for every
\((t,x)\in\mathbb R_+\times\mathbb R\).
Then the following statements hold for every \(T>0\).

\begin{itemize}

\myitem[(i)]{corollary_i}
The price process \((S_t)_{t\in[0,T]}\) is a martingale if
\[
\begin{cases}
\PP_0(\forall t\in[0,T],\ X_t\le0)=1, & \text{if }\rho A>0,\\
\PP_0(\forall t\in[0,T],\ X_t\ge0)=1, & \text{if }\rho A<0.
\end{cases}
\]

\myitem[(ii)]{corollary_ii}
For strict local martingality, we assume that
Assumptions~\myref{IV.2} and~\myref{V.1} hold
with threshold \(x^*\ge0\) and signs
\(s=\sgn(\rho)\), \(\varepsilon_s=\sgn(\rho A)\).
For every \(t_0\in[0,T)\) and \(t_1\in(t_0,T]\),
\((S_u)_{u\in[t_0,t_1]}\) is a strict local martingale if
\[
\begin{cases}
\PP_0(X_{t_0}>x^*)>0, & \text{if }\rho A>0,\\
\PP_0(X_{t_0}<-x^*)>0, & \text{if }\rho A<0.
\end{cases}
\]

\end{itemize}
\end{cor2}
Indeed, the martingality result follows from
Theorem~\ref{maintheorem}\myref{(Bi)}, restricted to $[0,T]$, using Assumption~\myref{III.1}
with \(s=\sgn(\rho)\), \(\varepsilon_s=\sgn(\rho A)\), and \(L_T=0\).
For part (ii), the same signs give
\(\varepsilon_s F_1(sz)=|A|z\), so the growth condition in Assumption~\myref{I.5} holds with $f_{1,s}=1$, and the monotonicity condition in \eqref{eq:power_branch_F1} of Assumption~\myref{V.1}
holds for every $z_0\in\mathbb R$.
If \(\PP_0(\varepsilon_s X_{t_0}>x^*)>0\), then, for some \(z_0<0\), we have
$\PP_0(
\varepsilon_s X_{t_0}>x^*,\
sY_{1,t_0}>z_0
)>0.$ The result therefore follows from Theorem~\ref{nonmartingality}.

\medskip

\begingroup

\begin{samepage}
\begin{lemma}[Accessibility of the explosive region]
\label{lem:accessibility_models}
Suppose that the assumptions of Theorem~\ref{nonmartingality} hold.
The stopping time \(\tau_*\) defined in Assumption~\myref{V.2} satisfies
\begin{align}
\forall T>0,\qquad \PP_0(\tau_*<T)&>0
\label{eq:accessibility_models}
\end{align}
in each of the following cases:
\begin{itemize}
\item the lognormal stochastic volatility, Quintic, and Volterra--Bergomi models;
\item the polynomial PDV model and specification~\eqref{eq:F1A}, with \(X_0>0\) 
and the choices of \(x^*,z_0\) in
Table~\ref{tab:verification_assumption_III5};
\item the Quadratic Volterra Heston+ model, with \(a_0>0\) and \(1\le p\le2\).
\end{itemize}
In particular, Assumption~\myref{V.2} holds in these cases.
\end{lemma}
\end{samepage}

\begin{proof}
For the lognormal, Quintic, and Volterra--Bergomi models,
\(Y_{1,t}\) is Gaussian with positive variance under \(\PP_0\)
for every \(t>0\). Since \(X_t=\beta(t)+Y_{1,t}\) and \(\varepsilon_s=s\), we have
\[
\forall t>0,\qquad
\PP_0\left(\varepsilon_s X_t>x^*,\ sY_{1,t}>z_0\right)
=\PP_0\left(sY_{1,t}>\max\{z_0,x^*-s\beta(t)\}\right)>0.
\]
Taking \(t=T/2\) gives \eqref{eq:accessibility_models}.
The proof for the three remaining models is given in
Appendix~\ref{app:accessibility_models}.
\end{proof}
\endgroup

For the Volterra--Guyon--Lekeufack model, the condition in
Corollary~\ref{corollary}\myref{corollary_ii}
is \(\PP_0(\sigma_{t_0}<0)>0\) for some \(t_0<T\).
Its validity depends on the kernels and the initial curves; explicit sufficient conditions for exponential kernels are given in
\cite[Theorem~3]{GuyonJeanninJourdain2026Positivity}.

\section{Girsanov reduction and preliminary results}

Lemma~\ref{martingality0} below relates the martingality of \(S\)
to non-explosion of the solution of \eqref{tildePDV}.

\begin{lemma}\label{martingality0}
Suppose that Assumption~\ref{standingassumption} holds
and that \(\sigma\) is locally bounded on \(\mathbb R_+\times\mathbb R\).
For a bounded stopping time \(\theta\) and \(T>0\), we then have
\[
\E_0\left[S_{\theta+T}\,\middle|\,\mathcal F_\theta\right]
=
S_\theta
\left(
1-\lim\limits_{n\to\infty}
\PP_0\left(\tau_{(\theta,n)}^{\widetilde X^\theta}\le T\,\middle|\,\mathcal F_\theta\right)
\right)
\quad\PP_0\text{-a.s.,}
\]
where \(\tau_{(\theta,n)}^{\widetilde X^\theta}\) is defined in
Assumption~\myref{standing_ii}.
Moreover, \((S_{\theta+t})_{t\in[0,T]}\) is a martingale under \(\PP_0\) with respect to \((\mathcal F_{\theta+t})_{t\in[0,T]}\) if and only if
$$
\lim\limits_{n\to\infty}
\PP_0\left(\tau_{(\theta,n)}^{\widetilde X^\theta}\le T\,\middle|\,\mathcal F_\theta\right)=0
\qquad\PP_0\text{-a.s.}
$$
\end{lemma}

\begin{proof}
We use a change-of-measure argument \cite{Sin1998,Jourdain2004,LionsMusiela2007,BernardCuiMcLeish2017}.

\medskip
\noindent\textit{Step 1: Change of measure.} We recall that
\(W_t^\rho=\rho W_t+\sqrt{1-\rho^2}\,W_t^\perp\). For \(n>0\), we set
$\tau_{(\theta,n)}^{X}:=\inf\{u\ge0:\ |X_{\theta+u}|\ge n\}.$
Since \(\theta\) is bounded, we fix a deterministic \(\bar T<\infty\)
such that \(\theta+T\le\bar T\) a.s. We define \(\mathbb Q_n\) on
\(\mathcal F_{\bar T}\) by
\[
\left.\frac{d\mathbb Q_n}{d\PP_0}\right|_{\mathcal F_{\bar T}}
:=\frac{S_{\theta+(T\wedge\tau_{(\theta,n)}^X)}}{S_\theta}.
\]
This ratio is the terminal value of the stochastic exponential with
respect to \(W^\rho\) whose integrand is
$\sigma(s,X_s)
\mathbf1_{\{\theta<s\le\theta+(T\wedge\tau_{(\theta,n)}^X)\}}.$ By local boundedness of \(\sigma\), this integrand is bounded on
\([0,\bar T]\). By Novikov's criterion, the exponential is a martingale.
It equals \(1\) at \(\theta\), so that
\[
\E_0\left[
\frac{S_{\theta+(T\wedge\tau_{(\theta,n)}^X)}}{S_\theta}
\,\middle|\,\mathcal F_\theta
\right]=1.
\]
Thus \(\mathbb Q_n\) and \(\PP_0\) agree on \(\mathcal F_\theta\).

\medskip
\noindent\textit{Step 2: Identification with \eqref{tildePDV}.} Girsanov's theorem shows that
\[
\left(
W_t-\rho\int_0^t\sigma(s,X_s)
\mathbf1_{\{\theta<s\le\theta+(T\wedge\tau_{(\theta,n)}^X)\}}\,ds
\right)_{t\in[0,\bar T]}
\]
is a Brownian motion under \(\mathbb Q_n\) and agrees with \(W\) on
\([0,\theta]\). Under \(\mathbb Q_n\), \((X,Y)\) solves \eqref{tildePDV},
driven by this Brownian motion, up to
\(\theta+(T\wedge\tau_{(\theta,n)}^X)\),
with history \((X_s,Y_s)_{s\in[0,\theta]}\).
Here \(\mathcal L_{\mathbb Q}(Z\mid\mathcal F_\theta)\) denotes the
conditional law of a random element \(Z\) under the probability measure
\(\mathbb Q\), given \(\mathcal F_\theta\).
Conditional uniqueness in law in Assumption~\myref{standing_ii} therefore gives
\begin{align}
\mathcal L_{\mathbb Q_n}\left(
(
X_{\theta+(u\wedge\tau_{(\theta,n)}^X)},
Y_{\theta+(u\wedge\tau_{(\theta,n)}^X)}
)_{u\in[0,T]}\,\middle|\,\mathcal F_\theta
\right)&=
\mathcal L_{\PP_0}\left(
(
\widetilde X^\theta_{\theta+(u\wedge
\tau_{(\theta,n)}^{\widetilde X^\theta})},
\widetilde Y^\theta_{\theta+(u\wedge
\tau_{(\theta,n)}^{\widetilde X^\theta})}
)_{u\in[0,T]}\,\middle|\,\mathcal F_\theta
\right).
\label{eq:stopped_conditional_state_law}
\end{align}
The definition of \(\mathbb Q_n\) and
\eqref{eq:stopped_conditional_state_law} give
\begin{align*}
\E_0\left[
\frac{S_{\theta+T}}{S_\theta}
\mathbf1_{\{\tau_{(\theta,n)}^X>T\}}
\,\middle|\,\mathcal F_\theta
\right]
&=\mathbb Q_n\left(
\tau_{(\theta,n)}^X>T\,\middle|\,\mathcal F_\theta
\right)=1-\PP_0\left(
\tau_{(\theta,n)}^{\widetilde X^\theta}\le T
\,\middle|\,\mathcal F_\theta
\right).
\end{align*}
By global existence in Assumption~\myref{standing_i}, we have
\(\mathbf1_{\{\tau_{(\theta,n)}^X>T\}}\underset{n\to\infty}{\uparrow}1\) a.s.
Conditional monotone convergence gives
\[
\E_0\left[
\frac{S_{\theta+T}}{S_\theta}
\,\middle|\,\mathcal F_\theta
\right]
=1-\lim\limits_{n\to\infty}\PP_0\left(
\tau_{(\theta,n)}^{\widetilde X^\theta}\le T
\,\middle|\,\mathcal F_\theta
\right).
\]
Since \(S\) is a nonnegative local martingale under \(\PP_0\), it is a supermartingale, and
\((S_{\theta+t})_{t\in[0,T]}\) is a martingale if and only if
\(\E_0[S_{\theta+T}\mid\mathcal F_\theta]=S_\theta\) \(\PP_0\)-a.s., which is equivalent to
$$
\lim\limits_{n\to\infty}
\PP_0\left(
\tau_{(\theta,n)}^{\widetilde X^\theta}\le T\,\middle|\,\mathcal F_\theta
\right)=0\qquad\PP_0\text{-a.s.}
$$
\end{proof}

\begin{proof}[Proof of Theorem~\ref{maintheorem}\myref{(A)}]
If \(\rho=0\), \(W^\perp\) is independent of the factor system.
For every \(t\ge0\), conditionally on \((X_u)_{u\in[0,t]}\),
the stochastic exponential \(S_t/S_0\) has a deterministic
square-integrable integrand and therefore has expectation \(1\).
Hence we have \(\E_0[S_t]=S_0\) for every \(t\ge0\), and the nonnegative
local martingale \(S\) is a martingale.
\end{proof}

\section{Martingality: Proof of Theorem~\ref{maintheorem}}

Throughout this section, \(T>0\) is fixed and \(\theta=0\).
We write \(\widetilde X\), \(\widetilde Y_1\), and \(\widetilde Y_2\)
instead of \(\widetilde X^0\), \(\widetilde Y_1^0\), and
\(\widetilde Y_2^0\), respectively.
We set
\[
\tau_n:=\inf\{t\ge0:\ |\widetilde X_t|\ge n\}.
\]
The constants \(C\), with or without subscripts, may depend on the initial
curves and change from line to line. They are independent of \(n\).

\begin{lemma}\label{lem:common_stopped_moment_criterion}
Suppose that Assumptions~\ref{assumption1} and~\ref{standingassumption} hold. The process defined by
\[
\forall t\in[0,T],\qquad
U_{1,t}^n:=\int_0^tK_1(t-s)\Gamma_1(s,\widetilde X_s)
\mathbf1_{\{s\le\tau_n\}}\,dW_s,
\]
admits a continuous modification. We use this continuous modification throughout the remainder of the paper. For some constant \(C_T\),
we assume that, for every \(n\in\mathbb N^*\), a.s.,
\begin{align}
\forall t\in[0,T],\qquad
\sup_{u\in[0,t]}|\widetilde X_{u\wedge\tau_n}|
&\le C_T\left(1+\sup_{u\in[0,t]}|U_{1,u}^n|^{f_1}
+\sup_{u\in[0,t]}|\widetilde Y_{2,u\wedge\tau_n}|^{f_2}\right).
\label{eq:common_X_control}
\end{align}
Then, for every \(r>0\), we have
\begin{align}
\sup_{n\in\mathbb N^*}
\E_0\left[\sup_{t\in[0,T]}|\widetilde X_{t\wedge\tau_n}|^r\right]
&\le C_{T,r},
\label{eq:common_stopped_moment_criterion}
\end{align}
and \((S_t)_{t\in[0,T]}\) is a martingale.
\end{lemma}

\begin{proof}
By H\"older's inequality, we only need to prove
\eqref{eq:common_stopped_moment_criterion} of Lemma~\ref{lem:common_stopped_moment_criterion} for \(r>0\) such that
\(rf_1>2\), \(r\delta_1f_1>2\), and \(rf_2\ge1\).
We fix such an \(r\) and \(n\in\mathbb N^*\), and set
\[
\forall t\in[0,T],\qquad
A_n(t):=\E_0\left[\sup_{u\in[0,t]}|\widetilde X_{u\wedge\tau_n}|^r\right].
\]
The function \(A_n\) is nondecreasing. Continuity of \(\widetilde X\) gives
\(A_n(t)\le(n\vee|\widetilde X_0|)^r<\infty\).

\medskip
\noindent\textbf{Step 1: Stochastic convolution estimate.}
By the polynomial growth bound on \(\Gamma_1\) in Assumption~\myref{I.5} and \(\gamma_1f_1\le1\), we have
\begin{align*}
\forall s\in[0,T],\qquad
\E_0\left[|\Gamma_1(s,\widetilde X_s)|^{rf_1}\mathbf1_{\{s\le\tau_n\}}\right]
&\le C_{T,r}\E_0\left[(1+|\widetilde X_s|^r)
\mathbf1_{\{s\le\tau_n\}}\right] \le C_{T,r}(1+A_n(s)).
\end{align*}
For \(u,v\in[0,T]\) with \(u<v\), the BDG inequality gives
\begin{align}
\E_0[|U_{1,v}^n-U_{1,u}^n|^{rf_1}]
&\le C_{T,r}\E_0\Biggl[
\Biggl(
\int_0^u|K_1(v-s)-K_1(u-s)|^2|\Gamma_1(s,\widetilde X_s)|^2\mathbf1_{\{s\le\tau_n\}}\,ds\notag\\*
&\qquad \qquad \qquad
+\int_u^vK_1(v-s)^2|\Gamma_1(s,\widetilde X_s)|^2\mathbf1_{\{s\le\tau_n\}}\,ds
\Biggr)^{\frac{rf_1}{2}}\Biggr]\notag\\*
&\le C_{T,r}(1+A_n(v))
\left(
\int_0^u|K_1(v-s)-K_1(u-s)|^2\,ds
+\int_u^vK_1(v-s)^2\,ds
\right)^{\frac{rf_1}{2}}\notag\\*
&\le C_{T,r}|v-u|^{r\delta_1f_1}(1+A_n(v)).
\label{eq:stochastic_convolution_increment_moment}
\end{align}
The second inequality uses weighted Jensen's inequality and
monotonicity of \(A_n\).
The last inequality follows from the increment bound on \(K_1\) in \eqref{eq:K1_increment_regular_assumption} of Assumption~\ref{assumption1}.
Since \(A_n(T)<\infty\), the increment estimate
\eqref{eq:stochastic_convolution_increment_moment}, with
\(r\delta_1f_1>2>1\), gives a continuous modification of \(U_1^n\)
by Kolmogorov's continuity theorem. (This ensures that \(\sup_{t\in[0,T]}|U_{1,t}^n|\) is well defined.)
Since \(U_{1,0}^n=0\), Lemma~\ref{lem:fractional_sobolev_sup} below
with exponent \(rf_1>2\) gives the first inequality below.
The second follows from \eqref{eq:stochastic_convolution_increment_moment}
by symmetry and Tonelli's theorem.
\begin{align}
\forall t\in[0,T],\qquad
\E_0\left[\sup_{u\in[0,t]}|U_{1,u}^n|^{rf_1}\right]
&\le C_{T,r}t\int_0^t\int_0^t
\frac{\E_0[|U_{1,v}^n-U_{1,u}^n|^{rf_1}]}
{|v-u|^3}\,du\,dv\notag\\*
&\le C_{T,r}t\int_0^t(1+A_n(v))
\int_0^v(v-u)^{r\delta_1f_1-3}\,du\,dv\notag\\*
&=\frac{C_{T,r}t}{r\delta_1f_1-2}
\int_0^t(1+A_n(v))v^{r\delta_1f_1-2}\,dv\notag\\*
&\le C_{T,r}\left(1+\int_0^tA_n(v)\,dv\right).
\label{eq:common_stochastic_maximal_bound}
\end{align}
The last inequality follows from \(r\delta_1f_1>2\) and
\(0\le v\le t\le T\).

\medskip
\noindent\textbf{Step 2: Stopped moments of \(\widetilde X\).}
For a nonnegative nondecreasing function \(\varphi\), nonnegativity of \(K_2\) gives
\begin{align*}
\forall 0\le u\le t\le T,\qquad
\int_0^uK_2(u-s)\varphi(s)\,ds
&=\int_0^uK_2(v)\varphi(u-v)\,dv\\
&\le\int_0^uK_2(v)\varphi(t-v)\,dv
\le\int_0^tK_2(v)\varphi(t-v)\,dv.
\end{align*}
Thus the convolution is nondecreasing.
For a continuous function \(g\) on \([0,T]\), we write
\(\|g\|_{\infty,[0,T]}:=\sup_{t\in[0,T]}|g(t)|\).
The polynomial growth bound on \(\Gamma_2\) in Assumption~\myref{I.5} gives
\[
\forall t\in[0,T],\qquad
\sup_{u\in[0,t]}|\widetilde Y_{2,u\wedge\tau_n}|
\le \|G_2\|_{\infty,[0,T]}
+C_T\int_0^tK_2(t-s)
\left(1+\sup_{v\in[0,s]}|\widetilde X_{v\wedge\tau_n}|^{\gamma_2}\right)ds.
\]
Weighted Jensen's inequality with exponent \(rf_2\ge1\),
\(K_2\in L^1([0,T])\), and \(\gamma_2f_2\le1\) give
\begin{align*}
\forall t\in[0,T],\qquad
\E_0\left[
\sup_{u\in[0,t]}|\widetilde Y_{2,u\wedge\tau_n}|^{rf_2}\right]
&\le C_{T,r}\left(1+\int_0^tK_2(t-s)
\E_0\left[\sup_{v\in[0,s]}|\widetilde X_{v\wedge\tau_n}|^{r\gamma_2f_2}\right]ds\right)\\
&\le C_{T,r}\left(1+\int_0^tK_2(t-s)A_n(s)\,ds\right).
\end{align*}
Combining this bound with \eqref{eq:common_X_control} of
Lemma~\ref{lem:common_stopped_moment_criterion} and
\eqref{eq:common_stochastic_maximal_bound}, we obtain
\[
\forall t\in[0,T],\qquad
A_n(t)\le C_{T,r}\left(
1+\int_0^t[1+K_2(t-s)]A_n(s)\,ds\right).
\]
The Volterra--Gr\"onwall lemma \cite[Lemma~2.2]{Zhang2010},
applied with the integrable kernel \(C_{T,r}(1+K_2)\), proves
\eqref{eq:common_stopped_moment_criterion} of
Lemma~\ref{lem:common_stopped_moment_criterion}.
Continuity of \(\widetilde X\) and Markov's inequality give
\[
\PP_0(\tau_n\le T)
\le n^{-r}\E_0\left[\sup_{t\in[0,T]}
|\widetilde X_{t\wedge\tau_n}|^r\right]
\le C_{T,r}n^{-r}\xrightarrow[n\to\infty]{}0.
\]
Lemma~\ref{martingality0} with \(\theta=0\) proves martingality on \([0,T]\).
\end{proof}

\subsection{Proof of Theorem~\ref{maintheorem}\myref{(Bii)}}

\leavevmode
\begingroup

\begin{lemma}[Volterra chain-rule]
\label{lem:convex_positive_resolvent_chain}
Suppose that Assumption~\myref{IV.1} holds. In particular,
\(K_1\in L^1([0,T])\) and \(K_1\ge0\).
For \(g\in L^\infty([0,T])\), we define \(B\) by
\[
\forall t\in[0,T],\qquad
B_t:=\int_0^t K_1(t-s)g_s\,ds.
\]
Let the function \(\Phi:\mathbb R\to\mathbb R_+\) belong to \(C^1(\mathbb R)\), be convex, and
satisfy \(\Phi(0)=0\). Then we have
\begin{align}
\forall t\in[0,T],\qquad
\Phi(B_t)
&\le
\int_0^t K_1(t-s)\Phi'(B_s)g_s\,ds .
\label{eq:convex_positive_resolvent_chain}
\end{align}
\end{lemma}
For \(a_{\rm res}=0\), \((k,l)=(\ell_{\rm res},K_1)\)
satisfies condition \((\mathcal{PC})\).
The convexity inequality and regularization argument of
Vergara--Zacher~\cite[Section~2, Lemma~2.1; Section~3, proof of Lemma~3.2]{VergaraZacher2017},
with \(u=B\), \(u_0=0\), and \(H=\Phi\), give
\eqref{eq:convex_positive_resolvent_chain}, since
\(\partial_t(\ell_{\rm res}*B)=g\).
The extension to \(a_{\rm res}>0\) follows immediately by the ordinary
chain rule for the atomic contribution: it is the \((\mathcal{PC}')\)
extended setting described in \cite[Section~1]{VergaraZacher2017}.
\endgroup

We set
\[
\forall t\in[0,T],\qquad
Z_t:=\widetilde Y_{1,t}-G_1(t).
\]
Under the resolvent condition in Assumption~\myref{IV.1}, \(Z\) need not be a semimartingale.
Thus Itô's formula cannot be used directly in Theorem~\ref{maintheorem}\myref{(Bii)}.
For \(n\in\mathbb N^*\), we write
\[
\forall t\in[0,T\wedge\tau_n],\qquad
Z_t=U_{1,t}^n+B_t,\qquad
B_t:=\rho\int_0^tK_1(t-s)
\Gamma_1(s,\widetilde X_s)\sigma(s,\widetilde X_s)\,ds.
\]
We estimate \(U_1^n\) by \eqref{eq:common_stochastic_maximal_bound}.

\begin{proof}[Proof of Theorem~\ref{maintheorem}\myref{(Bii)}]
\begingroup
For \(t\in[0,T]\), \(y,w\in\mathbb R\), we set
$x=\beta(t)+F_1(y)+F_2(w).$
By Assumption~\myref{I.5} and the boundedness of \(\beta\) on
\([0,T]\), there exist \(0<c\le C\) such that, for a sufficiently
large \(C_T\ge1\) and each \(s\in\{-1,1\}\), the conditions
\[
sy>0,\qquad |y|^{f_{1,s}}\ge C_T(1+|w|^{f_2})
\]
imply \(|y|\ge1\) and
\begin{align*}
c|y|^{f_{1,s}}
&\le\varepsilon_sF_1(y)\le C|y|^{f_{1,s}},\\
|\beta(t)|+|F_2(w)|
&\le C(1+|w|^{f_2})\le\frac{c}{2}|y|^{f_{1,s}}.
\end{align*}
With \(R_T\) as in Assumption~\myref{III.2}, increasing \(C_T\) if
necessary yields
\begin{align}
\varepsilon_s x
&\ge\varepsilon_sF_1(y)-|\beta(t)|-|F_2(w)|
\ge\frac{c}{2}|y|^{f_{1,s}}\ge R_T,\notag\\
\varepsilon_s x
&\le\varepsilon_sF_1(y)+|\beta(t)|+|F_2(w)|
\le2C|y|^{f_{1,s}}.
\label{eq:Bii_dominance_bounds}
\end{align}
By \eqref{eq:Bii_dominance_bounds}, we have \(|x|=\varepsilon_s x\ge R_T\).
Assumption~\myref{III.2} and \(|y|\ge1\) therefore give
\[
s\rho\Gamma_1(t,x)\sigma(t,x)
\le C_T\left(1+|x|^{\frac{1}{f_{1,s}}}\right)
\le C_T\left(1+(2C)^{\frac{1}{f_{1,s}}}|y|\right)\le C_T\left(1+(2C)^{\frac{1}{f_{1,s}}}\right)|y|.
\]
Since \((y^+)^{f_{1,+}}+(y^-)^{f_{1,-}}=|y|^{f_{1,s}}\) when
\(sy>0\), a further increase of \(C_T\) yields, for all
\((t,y,w)\in[0,T]\times\mathbb R^2\),
\begin{align}
(y^+)^{f_{1,+}}+(y^-)^{f_{1,-}}
&\ge C_T(1+|w|^{f_2})
\quad\Longrightarrow\quad
\operatorname{sgn}(y)\rho\Gamma_1(t,x)\sigma(t,x)
\le C_T|y|.
\label{eq:Bii_cone_drift}
\end{align}
\endgroup

\noindent\textbf{Step 1: Bound for \(Z^+\).}
We fix \(n\in\mathbb N^*\) and \(t\in[0,T]\). For each trajectory, we set
\[
m_t:=1+\sup_{u\in[0,t]}|U_{1,u}^n|^{f_1}
+\sup_{u\in[0,t]}|\widetilde Y_{2,u\wedge\tau_n}|^{f_2}.
\]
With \(C_T\) as in \eqref{eq:Bii_cone_drift}, we set
\[
A_T:=1+\|G_1\|_{\infty,[0,T]}
+\max\left\{C_T^{\frac{1}{f_{1,+}}},C_T^{\frac{1}{f_{1,-}}}\right\}.
\]
On \([0,t\wedge\tau_n]\), we have \(Z=U_1^n+B\), with
\[
\forall s\in[0,t\wedge\tau_n],\qquad
B_s=\rho\int_0^s K_1(s-u)
\Gamma_1(u,\widetilde X_u)\sigma(u,\widetilde X_u)\,du.
\]
Since \(m_t\ge1\) and \(f_{1,+}\le f_1\),
we have
\[
|U_{1,s}^n|^{f_{1,+}}
\le1+|U_{1,s}^n|^{f_1}\le m_t,
\]
so that
\begin{align} \label{etoile}
\forall s\in[0,t\wedge\tau_n],\qquad
|U_{1,s}^n|\vee|\widetilde Y_{2,s}|^{\frac{f_2}{f_{1,+}}}
\le m_t^{\frac{1}{f_{1,+}}}.
\end{align}
If \(B_s>A_Tm_t^{\frac{1}{f_{1,+}}}\), then, using \(m_t\ge1\),
\eqref{etoile}, and the definition of \(A_T\),
\(\widetilde Y_{1,s}=G_1(s)+U_{1,s}^n+B_s
>C_T^{\frac{1}{f_{1,+}}}m_t^{\frac{1}{f_{1,+}}}>0\).
Consequently, the definition of \(m_t\) gives
\[
(\widetilde Y_{1,s})^{f_{1,+}}>C_Tm_t
\ge C_T\left(1+|\widetilde Y_{2,s}|^{f_2}\right).
\]
Thus, if \(B_s>A_Tm_t^{\frac{1}{f_{1,+}}}\), the drift bound \eqref{eq:Bii_cone_drift},
with \(x=\widetilde X_s=\beta(s)+F_1(\widetilde Y_{1,s})+F_2(\widetilde Y_{2,s})\),
\(y=\widetilde Y_{1,s}\) and \(w=\widetilde Y_{2,s}\), gives
\begin{align}
\frac{\rho\Gamma_1(s,\widetilde X_s)\sigma(s,\widetilde X_s)}
     {m_t^{\frac{1}{f_{1,+}}}}
&\le C_T\frac{\widetilde Y_{1,s}}{m_t^{\frac{1}{f_{1,+}}}}\notag\\*
&\le C_T\left(1+\|G_1\|_{\infty,[0,T]}
+\frac{B_s}{m_t^{\frac{1}{f_{1,+}}}}\right)\notag\\*
&\le C_T(1+\|G_1\|_{\infty,[0,T]}+A_T)\left[1+
\left(\frac{B_s}{m_t^{\frac{1}{f_{1,+}}}}-A_T\right)^+\right],
\label{int2}
\end{align}
where we again used $m_t\ge1$ and \eqref{etoile}.
Continuity of \(\Gamma_1\sigma\) and \(\widetilde X\) gives
\[
\sup_{s\in[0,t\wedge\tau_n]}
|\Gamma_1(s,\widetilde X_s)\sigma(s,\widetilde X_s)|<\infty.
\]
For \(1<p\le2\), Lemma~\ref{lem:convex_positive_resolvent_chain}
therefore applies pathwise on \([0,t\wedge\tau_n]\) with
\begin{align*}
\forall x\in\mathbb R,\qquad \Phi(x)
&=\left[\left(\frac{x}{m_t^{\frac{1}{f_{1,+}}}}-A_T\right)^+\right]^p,\\
\forall s\in[0,t\wedge\tau_n],\qquad g_s
&=\rho\Gamma_1(s,\widetilde X_s)\sigma(s,\widetilde X_s),
\end{align*}
because, for \(p>1\), \(\Phi\) belongs to \(C^1(\mathbb R)\) and is convex and nonnegative,
with \(\Phi(0)=0\).
Since \(K_1\in L^1([0,T])\), \(B\) is bounded on
\([0,t\wedge\tau_n]\). 
Moreover, we have
\begin{align*}
\Phi'(B_s)
=\frac{p}{m_t^{\frac{1}{f_{1,+}}}}
\left[\left(\frac{B_s}{m_t^{\frac{1}{f_{1,+}}}}-A_T\right)^+\right]^{p-1}
\underset{p\downarrow1}{\longrightarrow}
\frac{\mathbf1_{\{B_s>A_Tm_t^{\frac{1}{f_{1,+}}}\}}}
{m_t^{\frac{1}{f_{1,+}}}}.
\end{align*}
Since \(pz^{p-1}\le2(1+z)\) for \(z\ge0\) and \(1<p\le2\),
boundedness of \(B,g\) gives an integrable bound after multiplication
by \(K_1(u-s)\).
Dominated convergence as \(p\downarrow1\), together with the drift bound \eqref{int2}, then gives  for all $u\in[0,t\wedge\tau_n]$,
\begin{align}
\left(\frac{B_u}{m_t^{\frac{1}{f_{1,+}}}}-A_T\right)^+
&\le\int_0^uK_1(u-s)
\mathbf1_{\{B_s>A_Tm_t^{\frac{1}{f_{1,+}}}\}} \frac{\rho\Gamma_1(s,\widetilde X_s)\sigma(s,\widetilde X_s)}
{m_t^{\frac{1}{f_{1,+}}}}\,ds\notag\\*
&\le C_T(1+\|G_1\|_{\infty,[0,T]}+A_T)\int_0^uK_1(u-s)\left[1+\left(\frac{B_s}{m_t^{\frac{1}{f_{1,+}}}}-A_T\right)^+\right]ds.
\label{chainrulefirst}
\end{align}
The Volterra--Gr\"onwall lemma \cite[Lemma~2.2]{Zhang2010} gives
\[
\sup_{u\in[0,t\wedge\tau_n]}
\left(\frac{B_u}{m_t^{\frac{1}{f_{1,+}}}}-A_T\right)^+\le C_T.
\]
Since for all $u\in[0,t\wedge\tau_n]$,
\begin{align*}
Z_u^+
&=(U_{1,u}^n+B_u)^+
\le \left(m_t^{\frac{1}{f_{1,+}}}+B_u\right)^+ \le m_t^{\frac{1}{f_{1,+}}}
\left[1+A_T+\left(\frac{B_u}{m_t^{\frac{1}{f_{1,+}}}}-A_T\right)^+\right],
\end{align*}
we obtain
\begin{align}
\sup_{u\in[0,t]}(Z_{u\wedge\tau_n}^+)^{f_{1,+}}&\le C_Tm_t.
\label{eq:Bii_positive_Z_bound}
\end{align}
\medskip
\noindent\textbf{Step 2: Bound for \(\widetilde X\).}
The argument for \(Z^-\) is obtained by replacing
\((Z,B,g,U_1^n,G_1,\widetilde Y_1)\) by
\((-Z,-B,-g,-U_1^n,-G_1,-\widetilde Y_1)\)
and \(f_{1,+}\) by \(f_{1,-}\) in the proof of
\eqref{eq:Bii_positive_Z_bound}.
The required drift bound follows from \eqref{eq:Bii_cone_drift}
with \(y=\widetilde Y_{1,s}<0\).
Thus, with the same \(m_t\), we have
\[
\sup_{u\in[0,t]}(Z_{u\wedge\tau_n}^-)^{f_{1,-}}\le C_Tm_t.
\]
Moreover, \(\widetilde Y_1=G_1+Z\) gives
\[
(\widetilde Y_{1,u})^\pm
\le\|G_1\|_{\infty,[0,T]}+Z_u^\pm.
\]
The tail bounds on \(F_1\) and its continuity therefore yield
\[
|F_1(\widetilde Y_{1,u})|
\le C_T\left(1+(Z_u^+)^{f_{1,+}}+(Z_u^-)^{f_{1,-}}\right).
\]
Since \(\beta,G_1\) are bounded on \([0,T]\), the polynomial growth bounds on \(F_1,F_2\) in Assumption~\myref{I.5} give
\begin{align}
\sup_{u\in[0,t]}|\widetilde X_{u\wedge\tau_n}|
&\le C_T\left(1+\sup_{u\in[0,t]}
((Z_{u\wedge\tau_n}^+)^{f_{1,+}}
+(Z_{u\wedge\tau_n}^-)^{f_{1,-}})
+\sup_{u\in[0,t]}|\widetilde Y_{2,u\wedge\tau_n}|^{f_2}\right)\notag\\*
&\le C_Tm_t.
\label{eq:Bii_X_control}
\end{align}
Hence \eqref{eq:common_X_control} of
Lemma~\ref{lem:common_stopped_moment_criterion} holds, which proves martingality
on \([0,T]\). Since \(T>0\) is arbitrary,
\((S_t)_{t\ge0}\) is a martingale.
\end{proof}

\subsection{Proof of Theorem~\ref{maintheorem}\myref{(Bi)}}

\begin{proof}[Proof of Theorem~\ref{maintheorem}\myref{(Bi)}]
We fix \(s\in\{-1,1\}\) as in Assumption~\myref{III.1}.
We use \(U_1^n\) as defined in Lemma~\ref{lem:common_stopped_moment_criterion}.
The constant \(L_T\) is the upper bound for \(\varepsilon_s X\) in Assumption~\myref{III.1}.
For every \(n\in\mathbb N^*\), the density of \(\mathbb Q_n\),
defined in the proof of Lemma~\ref{martingality0} with \(\theta=0\),
is strictly positive.
Thus \(\mathbb Q_n\) and \(\PP_0\) are equivalent on \(\mathcal F_T\),
and the upper bound on \(\varepsilon_s X\) in Assumption~\myref{III.1} holds
\(\mathbb Q_n\)-almost surely. The equality of the stopped laws in
\eqref{eq:stopped_conditional_state_law}, with \(\theta=0\), gives
\begin{align}
\PP_0\left(
\sup_{t\in[0,T]}\varepsilon_s\widetilde X_{t\wedge\tau_n}\le L_T
\right)&=1.
\label{eq:Bi_shifted_upper_bound}
\end{align}
\begingroup
By the bound on \(-s\rho\Gamma_1\sigma\) in Assumption~\myref{III.1}, the upper bound \eqref{eq:Bi_shifted_upper_bound}, and \(K_1\ge0\), we have
\begin{align}
\forall t\in[0,T\wedge\tau_n],\qquad
(s\widetilde Y_{1,t})^-
&\le |G_1(t)|+|U_{1,t}^n|+\int_0^t K_1(t-u)
(s\rho\Gamma_1(u,\widetilde X_u)\sigma(u,\widetilde X_u))^-\,du\notag\\*
&\le |U_{1,t}^n|+C_T\left(1+
\int_0^tK_1(t-u)|\widetilde X_u|^{\frac{1}{f_{1,-s}}}\,du\right),
\label{eq:Bi_negative_Y1_bound}
\end{align}
since \(G_1\) is bounded on \([0,T]\).
Since \(\varepsilon_sF_1(y)=\varepsilon_sF_1(s(sy))\), the bounds on \(F_1\) in Assumption~\myref{I.5} give
\[
\forall y\in\mathbb R,\qquad
(\varepsilon_s F_1(y))^-
\le C\left(1+((sy)^-)^{f_{1,-s}}\right).
\]
{Indeed, this negative part vanishes when \(sy\ge y_*\), by Assumption~\myref{I.5}, and is bounded on \(0\le sy\le y_*\) by continuity of \(F_1\).}
When \(sy<0\), the growth bound in the direction \(-s\) gives
\[
(\varepsilon_sF_1(y))^-
\le |F_1(y)|=|F_1(-s(-sy))|
\le C\left(1+(-sy)^{f_{1,-s}}\right).
\]
The upper bound \eqref{eq:Bi_shifted_upper_bound} therefore gives
\begin{align*}
\forall t\in[0,T\wedge\tau_n],\qquad
|\widetilde X_t|
&=\varepsilon_s\widetilde X_t+2(\varepsilon_s\widetilde X_t)^-\\*
&\le |L_T|+2|\beta(t)|
 +2(\varepsilon_s F_1(\widetilde Y_{1,t}))^-
 +2|F_2(\widetilde Y_{2,t})|\\*
&\le C_T\left(1+
((s\widetilde Y_{1,t})^-)^{f_{1,-s}}
+|\widetilde Y_{2,t}|^{f_2}\right).
\end{align*}
Raising both sides to the power \(\frac{1}{f_{1,-s}}\) and using the bound \eqref{eq:Bi_negative_Y1_bound} on
\((s\widetilde Y_{1,t})^-\), we obtain
\[
\forall t\in[0,T\wedge\tau_n],\qquad
|\widetilde X_t|^{\frac{1}{f_{1,-s}}}
\le C_T\left(1+|U_{1,t}^n|+|\widetilde Y_{2,t}|^{\frac{f_2}{f_{1,-s}}}
+\int_0^tK_1(t-u)|\widetilde X_u|^{\frac{1}{f_{1,-s}}}\,du\right).
\]
For each trajectory and \(t\in[0,T]\), the Volterra--Gr\"onwall lemma \cite[Lemma~2.2]{Zhang2010}, applied on
\([0,t\wedge\tau_n]\), gives
\[
\sup_{u\in[0,t]}|\widetilde X_{u\wedge\tau_n}|^{\frac{1}{f_{1,-s}}}
\le C_T\left(1+\sup_{u\in[0,t]}|U_{1,u}^n|
+\sup_{u\in[0,t]}|\widetilde Y_{2,u\wedge\tau_n}|^{\frac{f_2}{f_{1,-s}}}\right).
\]
Raising this bound to the power \(f_{1,-s}\) and using \(f_{1,-s}\le f_1\) gives
\begin{align}
\forall t\in[0,T],\qquad
\sup_{u\in[0,t]}|\widetilde X_{u\wedge\tau_n}|
&\le C_T\left(1+\sup_{u\in[0,t]}|U_{1,u}^n|^{f_{1,-s}}
+\sup_{u\in[0,t]}|\widetilde Y_{2,u\wedge\tau_n}|^{f_2}\right)\notag\\*
&\le C_T\left(1+\sup_{u\in[0,t]}|U_{1,u}^n|^{f_1}
+\sup_{u\in[0,t]}|\widetilde Y_{2,u\wedge\tau_n}|^{f_2}\right).
\label{eq:Bi_X_by_Xi_Y2}
\end{align}
\endgroup
Hence \eqref{eq:common_X_control} of
Lemma~\ref{lem:common_stopped_moment_criterion} holds, which proves martingality
on \([0,T]\). Since \(T>0\) is arbitrary,
\((S_t)_{t\ge0}\) is a martingale.
\end{proof}

\section{Strict local martingality: Proof of Theorem~\ref{nonmartingality}}
\label{subsec:loss_martingality}

We fix \(T>0\), a bounded stopping time \(\theta\), and a deterministic
\(\bar T<\infty\) such that \(\theta+T\le\bar T\) a.s.
We write
\((\widetilde X,\widetilde Y)=(\widetilde X^\theta,\widetilde Y^\theta)\)
for the solution of \eqref{tildePDV}.
Throughout this section, \(s\in\{-1,1\}\) is one possible sign in Assumption~\myref{V.1}, and \(\varepsilon_s\) is fixed by Assumption~\myref{I.5}.

\subsection{Outline of the proof and preliminary lemmas}

\begin{notation}
We set
\[
\forall t\in[0,+\infty),\qquad
\hypertarget{fig:explosion:def-Zbar}{\bar Z_t}
:=s(\widetilde Y_{1,t}-G_1(t)).
\]
\end{notation}
By Assumption~\myref{standing_ii}, \(\widetilde X\) and \(\widetilde Y\)
are continuous and adapted before explosion. Since \(G_1\) is continuous,
the same holds for \(\bar Z\). The first hitting and exit
times used below are therefore stopping times.
The proof is summarized in Figure~\ref{fig:explosion-proof-structure}.

\begin{figure}[H]
\centering
\resizebox{0.98\textwidth}{!}{%
\begin{tikzpicture}[
    phase/.style={
        draw,
        thick,
        rounded corners,
        align=center,
        text width=3.75cm,
        minimum height=1.55cm,
        inner sep=5pt,
        font=\footnotesize
    },
    phasebrown/.style={
        phase,
        very thick
    },
    phasedrift/.style={
        phase,
        very thick
    },
    tool/.style={
        draw,
        dashed,
        rounded corners,
        align=center,
        text width=4.15cm,
        minimum height=1.10cm,
        inner sep=4pt,
        font=\footnotesize
    },
    commontool/.style={
        tool,
        text width=7.35cm
    },
    arrow/.style={
        -{Stealth[length=2.7mm,width=2mm]},
        thick,
        shorten >=4pt,
        shorten <=4pt
    },
    usearrow/.style={
        -{Stealth[length=2.2mm,width=1.7mm]},
        dashed,
        thick,
        shorten >=3pt,
        shorten <=3pt
    }
]


\node[phase] (Start) at (0,0) {
\textbf{1. Initial conditions}\\[2pt]
\(\varepsilon_s\widetilde X_\theta>x^*\)\\[1pt]
\(s\widetilde Y_{1,\theta}>z_0\)
};

\node[phasebrown] (Reach) at (5.1,0) {
\textbf{2. Reaching level \(n_0\)}\\[2pt]
By Lemma~\ref{lem:reach_from_positive_X_linear} below,\\
\(\bar Z\) reaches level \(n_0\)\\
and \(\mathcal D_{n_0}\) occurs
};

\node[phasedrift] (Double) at (10.2,0) {
\textbf{3. From \(n\) to \(2n\)}\\[2pt]
From level \(n\) to level \(2n\)\\
within time \(h_n\):\\[1pt]
conditional probability bound\\[1pt]
\eqref{probdoubling} of Lemma~\ref{lem:dyadic_step_Y_linear} below
};
\node[phase] (Explode) at (15.3,0) {
\textbf{4. Explosion}\\[2pt]
Repeating the doubling step at levels \(2^k n_0\)\\
with\\[1pt]
\(\displaystyle\sum_k h_{2^k n_0}<\infty\)
};

\draw[arrow] (Start.east) -- (Reach.west);
\draw[arrow] (Reach.east) -- (Double.west);
\draw[arrow] (Double.east) -- (Explode.west);


\node[tool] (ToolStart) at (0,2.85) {
\(G_1,G_2\) continuous,\\
\(\bar Z,\widetilde Y_2,\widetilde X\) bounded on \([0,\theta]\)
};

\node[commontool] (ToolSemi) at (7.65,3.05) {
Used in the reaching and doubling steps:\\
semimartingale representation of \(\bar Z\), separating\\
Brownian noise, positive drift \(s\rho K_1(0)\Gamma_1\sigma\), and memory drift term\\[1pt]
Lemma~\ref{lem:stopped_SDE_shifted_Z} below
};

\node[tool] (ToolIterate) at (15.3,2.85) {
Time bound for each step\\[1pt]
\(\displaystyle h_n=c_0n^{1-f_{1,s}\chi}\)
};

\node[commontool] (ToolSafety) at (7.65,-2.75) {
Increment bounds and a bound on \(F_2(\widetilde Y_2)\) over short intervals:\\[1pt]
\eqref{eq:Y2_increment_bound} and \eqref{eq:Y2_absorption_bound}
};


\draw[usearrow] (ToolStart.south) -- (Start.north);

\draw[usearrow] (ToolSemi.south west) -- (Reach.north);
\draw[usearrow] (ToolSemi.south east) -- (Double.north);

\draw[usearrow] (ToolIterate.south) -- (Explode.north);

\draw[usearrow] (ToolSafety.north west) -- (Reach.south);
\draw[usearrow] (ToolSafety.north east) -- (Double.south);

\end{tikzpicture}
}
\caption{Structure of the explosion proof.}
\label{fig:explosion-proof-structure}
\end{figure}
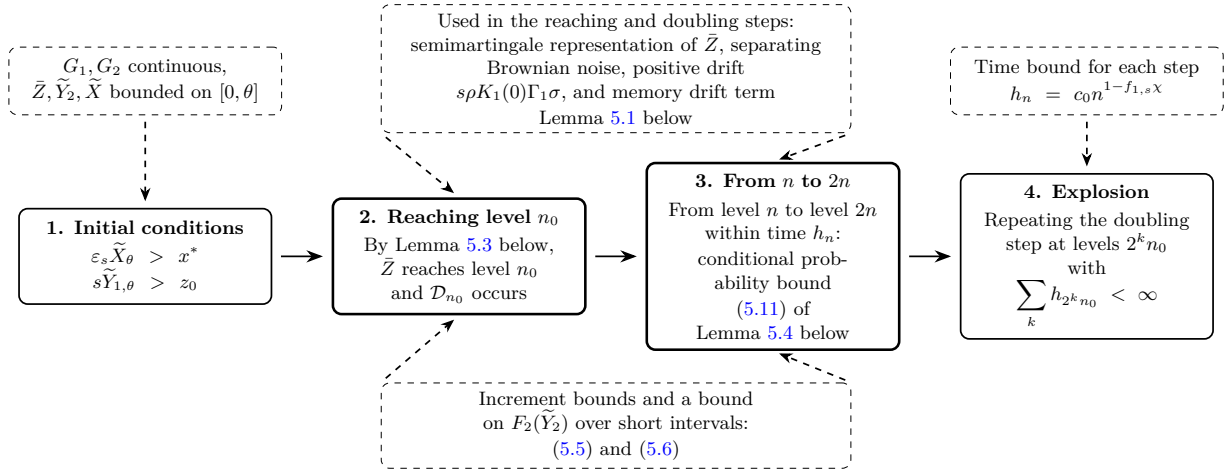
We use the resolvent of \(K_1\) to obtain a semimartingale
representation of \(\bar Z\).
The bound in Lemma~\ref{lem:stopped_SDE_shifted_Z}\myref{stopped_SDE_shifted_Z_ii} below
controls the memory drift term in
\eqref{eq:resolvent_semimartingale_Zbar}. Following \cite{BondiPulido2024}, we estimate the memory drift term after stopping the solution. Under the kernel regularity in Assumption~\myref{IV.2}, \(\mathcal R_1\) denotes the resolvent of the first kind constructed in Appendix~\ref{app:proof-stopped-SDE-shifted-Z}, and
\begin{align}
r_{\rm mem}&:=K_1'*\mathcal R_1
\in BV_{\mathrm{loc}}(\mathbb R_+).
\label{eq:smooth-memory-kernel}
\end{align}
We define the signed measure \(dr_{\rm mem}\) by
\(dr_{\rm mem}((a,b]):=r_{\rm mem}(b)-r_{\rm mem}(a)\) for \(0\le a<b\),
with no atom at zero; \(|dr_{\rm mem}|\) denotes its total variation.

\begin{lemma}\label{lem:stopped_SDE_shifted_Z}
Suppose that Assumption~\myref{IV.2} holds.
For \(R\in\mathbb N^\ast\), we set
\[
\tau_R^{\widetilde X}:=\inf\{t\ge0:\ |\widetilde X_t|\ge R\}.
\]
We use \(\mathcal R_1\) from \eqref{eq:smooth-canonical-resolvent}
and \(r_{\rm mem}\) from \eqref{eq:smooth-memory-kernel}. Then the following statements hold:
\begin{itemize}

\myitem[\textup{(i)}]{stopped_SDE_shifted_Z_i}
The stopped process
\((\bar Z_{t\wedge\tau_R^{\widetilde X}})_{t\ge0}\) is a semimartingale and satisfies
\begin{align}
\forall t \in [0,\tau_R^{\widetilde X}], \qquad \bar Z_t
&=
K_1(0)
\left(
s
\int_0^t
\Gamma_1(u,\widetilde X_u)\,dW_u
+
s\rho
\int_0^t
\mathbf{1}_{\{u\ge\theta\}}
\Gamma_1(u,\widetilde X_u)\sigma(u,\widetilde X_u)\,du
\right)
\notag\\
&\quad
+
\int_0^t
(
r_{\rm mem}(0)\bar Z_u+(dr_{\rm mem}\ast \bar Z)(u)
)\,du .
\label{eq:resolvent_semimartingale_Zbar}
\end{align}
\myitem[\textup{(ii)}]{stopped_SDE_shifted_Z_ii}
\begin{align}
\forall t\in[0,\bar T\wedge\tau_R^{\widetilde X}],\qquad
|r_{\rm mem}(0)\bar Z_t+(dr_{\rm mem}*\bar Z)(t)|
&\le(|r_{\rm mem}(0)|+|dr_{\rm mem}|([0,\bar T]))
\sup_{u\in[0,t]}|\bar Z_u|.
\label{eq:memory_pointwise_bound}
\end{align}

\end{itemize}
\end{lemma}
The proof is given in \hyperref[app:proof-stopped-SDE-shifted-Z]{Appendix~\ref*{app:proof-stopped-SDE-shifted-Z}}.
\begin{lemma}\label{lem:accessibility_tube}
Let \(Z\) be a continuous adapted process,
\(\tau\) a stopping time, and \(a,b\) predictable processes satisfying
\[
\forall s\in[0,+\infty),\qquad
Z_{s\wedge\tau}-Z_0
=\int_0^{s\wedge\tau}b_u\,du
+\int_0^{s\wedge\tau}a_u\,dW_u.
\]
We assume that there are constants \(h>0\) and \(0<c\le C<\infty\)
such that \(\tau\le h\) almost surely and
\[
\PP_0(
\forall s\in[0,\tau),\quad
|b_s|\le C,\quad c\le|a_s|\le C
)=1.
\]
Then, for every \(\delta>0\) and every \(p\in C^1([0,h])\)
with \(p(0)=0\), we have
\begin{align}
\PP_0\left(
\sup_{s\in[0,\tau]}|Z_s-Z_0-p(s)|<\delta
\right)&>0.
\label{eq:stopped_accessibility_tube}
\end{align}
\end{lemma}

\begin{proof}
We define the progressively measurable process \(u\) by
\(u_s:=\frac{p'(s)-b_s}{a_s}\) for \(s<\tau\) and \(u_s:=0\) afterwards.
It is bounded. Girsanov's theorem therefore gives a probability measure
\(\mathbb Q\) equivalent to \(\PP_0\), under which
\(Z_{s\wedge\tau}-Z_0-p(s\wedge\tau)\) is a continuous martingale
with quadratic variation bounded by \(C^2h\).
By the Dambis--Dubins--Schwarz theorem, there is a Brownian motion
\(B\) under \(\mathbb Q\), with an independent extension if necessary, such that
\[
\mathbb Q\left(
\sup_{s\in[0,\tau]}|Z_s-Z_0-p(s)|<\delta
\right)
\ge\mathbb Q\left(
\sup_{v\in[0,C^2h]}|B_v|<\delta
\right)>0.
\]
Equivalence of the measures gives \eqref{eq:stopped_accessibility_tube} of Lemma~\ref{lem:accessibility_tube}.
\end{proof}
The same conclusion holds conditionally after a bounded stopping time \(\theta\), with strictly positive conditional probability and \(\mathcal F_\theta\)-measurable parameters. The proof uses the shifted Brownian increments and localization of the bounds on \(h\), \(C\), \(c^{-1}\), and \(\|p'\|_\infty\) in the change of measure.

\subsection{Proof of explosion}

We work conditionally on \(\mathcal F_\theta\).
 Under Assumption~\ref{assumption1}, for
\(u,v\in[0,\bar T]\) with \(u\le v\), before explosion,
the polynomial growth bound on \(\Gamma_2\) in Assumption~\myref{I.5} and the increment bound on \(K_2\) in \eqref{eq:K2_increment_regular_assumption} of Assumption~\ref{assumption1} give
\begin{align}
|\widetilde Y_{2,v}-\widetilde Y_{2,u}|
&\le |G_2(v)-G_2(u)|\notag\\
&\quad+\int_0^u|K_2(v-r)-K_2(u-r)|
|\Gamma_2(r,\widetilde X_{r})|\,dr\notag\\
&\quad+\int_u^v K_2(v-r)|\Gamma_2(r,\widetilde X_{r})|\,dr\notag\\
&\le |G_2(v)-G_2(u)|
+C_{\bar T}\left(1+\sup_{r\in[0,v]}|\widetilde X_r|^{\gamma_2}\right)
|v-u|^{\delta_2}.
\label{eq:Y2_increment_bound}
\end{align}
For \(\theta\le u\le t\le\theta+T\),
the increment bound \eqref{eq:Y2_increment_bound} gives
\[
|\widetilde Y_{2,u}|
\le |\widetilde Y_{2,\theta}|+2\|G_2\|_{\infty,[0,\bar T]}
+C(t-\theta)^{\delta_2}
\left(1+\sup_{r\in[0,t]}|\widetilde X_r|^{\gamma_2}\right).
\]
Since \(\gamma_2f_2\le1\), the polynomial growth bound on \(F_2\) in Assumption~\myref{I.5} gives
{
\begin{align*}
\sup_{u\in[\theta,t]}|F_2(\widetilde Y_{2,u})|
&\le B+C(t-\theta)^{\delta_2f_2}
\left(1+\sup_{r\in[0,t]}|\widetilde X_r|^{\gamma_2f_2}\right)\\
&\le B+2CT^{\delta_2f_2}
+C(t-\theta)^{\delta_2f_2}
\sup_{r\in[0,t]}|\widetilde X_r|.
\end{align*}
Here \(B<\infty\) depends only on the history up to \(\theta\) and the
fixed data, while \(C\) is independent of that history and of \(t\).
We may increase \(B\) below.
After increasing \(B\) to absorb \(2CT^{\delta_2f_2}\), we obtain
}
\begin{align}
\sup_{u\in[\theta,t]}|F_2(\widetilde Y_{2,u})|
&\le B+C(t-\theta)^{\delta_2f_2}
\sup_{r\in[0,t]}|\widetilde X_r|.
\label{eq:Y2_absorption_bound}
\end{align}

\begin{notation}
We set \(\vartheta_n:=\inf\{t\ge\theta:\bar Z_t\ge n\}\). We also define the event at level \(n\)
{\begin{align}
\hypertarget{fig:explosion:def-Dn}{\mathcal D_n}
&:=
\{
\vartheta_n<+\infty
\}
\cap
\left\{
\inf_{u\in[\theta,\vartheta_n]}s\widetilde Y_{1,u}>z_0
\right\}.
\label{eq:dyadic_admissible_event}
\end{align}}
\end{notation}

\begin{lemma}\label{lem:reach_from_positive_X_linear}
Suppose that the assumptions of Theorem~\ref{nonmartingality} hold.
For every \(T_*\in(0,T]\) and every integer \(n\ge1\),
we have
\[
\PP_0\left(
\mathcal D_n\cap\left\{\vartheta_n\le\theta+{\frac{T_*}{2}}\right\}
\,\middle|\,\mathcal F_\theta
\right)>0
\quad\text{a.s.\ on}\quad
{\left\{n>1+\sup_{u\in[0,\theta]}|\bar Z_u|\right\}
\cap\{\varepsilon_s X_\theta>x^*,
sY_{1,\theta}>z_0\}}.
\]
\end{lemma}

\begingroup

\begin{proof}
We write
\((\widetilde X,\widetilde Y)
=(\widetilde X^\theta,\widetilde Y^\theta)\)
and recall that
\((\widetilde X_\theta,\widetilde Y_\theta)
=(X_\theta,Y_\theta)\).
We work conditionally on \(\mathcal F_\theta\), on
\(\left\{n>1+\sup_{u\in[0,\theta]}|\bar Z_u|\right\}
\cap\{\varepsilon_s X_\theta>x^*,
sY_{1,\theta}>z_0\}\).
We use Lemma~\ref{lem:accessibility_tube} to make \(\bar Z\)
reach \(n\) while keeping \(\varepsilon_s\widetilde X>x^*\)
and \(s\widetilde Y_1>z_0\).
Continuity of \(F_1,F_2\) gives \(\delta\in(0,\frac{1}{2}]\) such that
\(sY_{1,\theta}-3\delta>z_0\) and, for
\(w\in[Y_{2,\theta}-\delta,
Y_{2,\theta}+\delta]\),
\[
\varepsilon_s\beta(\theta)-\delta
+\varepsilon_s F_1(Y_{1,\theta}-3s\delta)
+\varepsilon_s F_2(w)
\ge\frac{\varepsilon_s X_\theta+x^*}{2}.
\]
For \(h\in(0,\frac{T_*}{2}]\), we set
\[
\tau=h\wedge\inf\{r\ge0:
\bar Z_{\theta+r}\notin(\bar Z_\theta-2\delta,n)
\ \text{or}\quad
|\widetilde Y_{2,\theta+r}-Y_{2,\theta}|\ge\delta
\}.
\]
Since \(n>1+\sup_{u\in[0,\theta]}|\bar Z_u|\) and
\(\delta\le\frac{1}{2}\), the definition of \(\tau\) gives
\[
\sup_{u\in[0,\theta+\tau]}|\bar Z_u|\le n,
\qquad
\sup_{u\in[0,\theta+\tau]}|\widetilde Y_{2,u}|
<1+\sup_{u\in[0,\theta]}|\widetilde Y_{2,u}|.
\]
Since \(\beta,G_1\) are bounded on \([0,\bar T]\) and \(F_1,F_2\)
are continuous, we have \(\sup_{u\in[0,\theta+\tau]}|\widetilde X_u|\le C\),
where \(C\) does not depend on \(h\).
Indeed, \(\widetilde Y_1=G_1+s\bar Z\) and the bounds on \(\bar Z\) and \(\widetilde Y_2\) give
\begin{align*}
\sup_{u\in[0,\theta+\tau]}|\widetilde X_u|
&\le\|\beta\|_{\infty,[0,\bar T]}
+\max_{|y|\le n+\|G_1\|_{\infty,[0,\bar T]}}|F_1(y)|+\max_{|w|\le1+\sup_{v\in[0,\theta]}|\widetilde Y_{2,v}|}|F_2(w)|.
\end{align*}
We choose \(h\) small enough that
\[
\sup_{r\in[0,h]}
(|\beta(\theta+r)-\beta(\theta)|
\vee|G_1(\theta+r)-G_1(\theta)|)\le\delta
\]
and
\[
\sup_{r\in[0,h]}|G_2(\theta+r)-G_2(\theta)|
+C_{\bar T}(1+C^{\gamma_2})h^{\delta_2}<\delta.
\]
The increment bound on \(\widetilde Y_2\) in \eqref{eq:Y2_increment_bound} gives
\[
\forall r\in[0,\tau],\qquad
|\widetilde Y_{2,\theta+r}-Y_{2,\theta}|
\le |G_2(\theta+r)-G_2(\theta)|
+C_{\bar T}(1+C^{\gamma_2}){r}^{\delta_2}<\delta.
\]
Thus, before \(h\), only \(\bar Z\) can cause the stop at \(\tau\).
By the definition of \(\tau\), the increment bound on \(G_1\), and
monotonicity of \(z\mapsto\varepsilon_s F_1(sz)\) on \([z_0,+\infty)\), we have
\begin{align}
\forall r\in[0,\tau],\qquad
s\widetilde Y_{1,\theta+r}
&=sY_{1,\theta}
+\bar Z_{\theta+r}-\bar Z_\theta
+s\bigl(G_1(\theta+r)-G_1(\theta)\bigr)\notag\\
&\ge{s}Y_{1,\theta}-2\delta-\delta
=sY_{1,\theta}-3\delta>z_0,\notag\\
\varepsilon_s\widetilde X_{\theta+r}
&{}\ge\varepsilon_s\beta(\theta)-\delta
+\varepsilon_s F_1(Y_{1,\theta}-3s\delta)
+\varepsilon_s F_2(\widetilde Y_{2,\theta+r})\notag\\
&\ge\frac{\varepsilon_s X_\theta+x^*}{2}.
\label{eq:strip_condition_linear}
\end{align}
Since \(\Gamma_1,\sigma\) are bounded on
\([0,\bar T]\times[-C,C]\),
we may increase \(C\) so that
\(|\sigma(\theta+r,\widetilde X_{\theta+r})|\le C\) for \(r\in[0,\tau]\).
The drift lower bound \eqref{eq:explosive_drift_growth} of Assumption~\myref{V.1} and the lower bound on
\(\varepsilon_s\widetilde X\) in \eqref{eq:strip_condition_linear} give
\begin{align*}
\forall r\in[0,\tau],\qquad
C|\Gamma_1(\theta+r,\widetilde X_{\theta+r})|
&\ge{s\rho}\Gamma_1(\theta+r,\widetilde X_{\theta+r})
\sigma(\theta+r,\widetilde X_{\theta+r}) \ge c_{\bar T}
\left(\frac{\varepsilon_s X_\theta+x^*}{2}\right)^\chi>0.
\end{align*}
Since \(K_1(0)>0\), the last lower bound and boundedness of
\(\Gamma_1\) give constants \(0<c_*\le C_*<\infty\),
independent of \(h\), such that
\begin{align}
\forall r\in[0,\tau],\qquad
c_*&\le K_1(0)|\Gamma_1(\theta+r,\widetilde X_{\theta+r})|\le C_*.
\label{eq:reach_diffusion_bounds}
\end{align}
The identity in \eqref{eq:resolvent_semimartingale_Zbar} and
the bound in \eqref{eq:memory_pointwise_bound} of
Lemma~\ref{lem:stopped_SDE_shifted_Z} also bound the drift of
\(\bar Z_{\theta+\cdot}\) on \([0,\tau]\).
Lemma~\ref{lem:accessibility_tube}, with
\[
\forall r\in[0,h],\qquad
p(r)=\frac{r}{h}(n-\bar Z_\theta+2\delta),
\]
gives
\[
\PP_0\left(
\sup_{r\in[0,\tau]}|\bar Z_{\theta+r}-\bar Z_\theta-p(r)|<\delta
\,\middle|\,\mathcal F_\theta\right)>0.
\]
Here we have \(p\ge0\) and \(\bar Z_\theta+p(h)=n+2\delta\).
On
\[
\left\{
\sup_{r\in[0,\tau]}
|\bar Z_{\theta+r}-\bar Z_\theta-p(r)|<\delta
\right\},
\]
we have
\[
\forall r\in[0,\tau],\qquad
\bar Z_{\theta+r}>\bar Z_\theta+p(r)-\delta
\ge\bar Z_\theta-\delta>\bar Z_\theta-2\delta.
\]
Still on this event, \(\vartheta_n>\theta+\tau\) would imply \(\tau=h\) and
\(\bar Z_{\theta+h}>\bar Z_\theta+p(h)-\delta=n+\delta>n\).
This contradicts the definition of \(\vartheta_n\).
Therefore, we have \(\vartheta_n=\theta+\tau\le\theta+\frac{T_*}{2}\).
The bound on \(s\widetilde Y_1\) in \eqref{eq:strip_condition_linear} gives \(\mathcal D_n\).
\end{proof}
\endgroup

\medskip

For a constant \(c_0>0\), chosen in
Lemma~\ref{lem:dyadic_step_Y_linear} below, we set
\begin{align}
\hypertarget{fig:explosion:def-tn}{h_n}
&:=\hypertarget{fig:explosion:def-c0}{c_0}n^{1-f_{1,s}\chi},
\qquad \tau_n^{\rm step}
:=h_n\wedge\inf\left\{
 r\ge0:
 \bar Z_{\vartheta_n+r}\notin(\frac{n}{2},2n)
\right\}.
\label{eq:def-dyadic-time}
\end{align}
We set \(\tau_n^{\rm step}:=0\) on \(\{\vartheta_n=\infty\}\).
Events involving \((\widetilde X,\widetilde Y)\) at \(\vartheta_n\) include
the condition \(\vartheta_n<\infty\).

The times used as \(\bar Z\) moves from each level \(n\) to \(2n\)
are listed in Table~\ref{tab:explosion-stopping-times}.
\begin{table}[H]
\centering
\caption{Times used in the explosion proof.
}
\label{tab:explosion-stopping-times}

\scriptsize
\renewcommand{\arraystretch}{1.25}
\setlength{\tabcolsep}{2pt}

\begin{tabularx}{\linewidth}{
|>{\hsize=0.640000\hsize\linewidth=\hsize\RaggedRight\arraybackslash}X
|>{\hsize=0.480000\hsize\linewidth=\hsize\RaggedRight\arraybackslash}X
|>{\hsize=0.640000\hsize\linewidth=\hsize\RaggedRight\arraybackslash}X
|>{\hsize=2.240000\hsize\linewidth=\hsize\RaggedRight\arraybackslash}X|}
\hline
\textbf{Symbol} & \textbf{Process} & \textbf{Stopping condition} & \textbf{Role} \\
\hline

\(\tau_R^{\widetilde X}\)
& \(|\widetilde X|\)
& \(R\)
& Stops the process for the semimartingale estimates. \\
\hline\(\vartheta_n\)
& \(\bar Z\)
& Level \(n\)
& Time when \(\bar Z\) reaches \(n\). \\
\hline\(\tau_n^{\rm step}\)
& \(\bar Z\)
& \(2n,\frac{n}{2},h_n\)
& Time after \(\vartheta_n\) until \(\bar Z\) reaches \(\frac{n}{2}\) or \(2n\), or \(h_n\) if neither level is reached earlier. \\
\hline
\end{tabularx}
\end{table}

\begin{lemma}\label{lem:dyadic_step_Y_linear}
Suppose that the assumptions of Theorem~\ref{nonmartingality} hold.
There exist deterministic constants \(T_*\in(0,T]\) and \(c_0,C>0\) such that,
for some a.s.\ finite \(\mathcal F_\theta\)-measurable integer
\(N_\theta\) and every integer \(n\ge1\),
\begin{align}
\PP_0\left(
\mathcal D_{2n}\cap
\{\vartheta_{2n}\le\vartheta_n+h_n\}
\,\middle|\,\mathcal F_{\vartheta_n}
\right)
&\ge1-Ch_n
\label{probdoubling}
\end{align}
a.s.\ on
\({\{n\ge N_\theta\}\cap}
\mathcal D_n\cap\{\vartheta_n+h_n\le\theta+T_*\}\),
where \(\mathcal D_n\) and \(h_n\) are defined in \eqref{eq:dyadic_admissible_event} and \eqref{eq:def-dyadic-time}, respectively.
\end{lemma}

\begin{proof}
We work on
\({\{n\ge N_\theta\}\cap}
\mathcal D_n\cap\{\vartheta_n+h_n\le\theta+T_*\}\),
with \(T_*\in(0,T]\) chosen below.
The integer \(N_\theta\) is chosen large enough for all
the bounds below; it depends on the finite quantities
\(B\) from \eqref{eq:Y2_absorption_bound}, \(\|G_1\|_{\infty,[0,\bar T]}\), and
\(\sup_{u\in[0,\theta]}|\bar Z_u|\).
We first stop at \(\tau_R^{\widetilde X}\) and prove the estimates
with constants independent of \(R\).
By the definition of \(\tau_n^{\rm step}\) in \eqref{eq:def-dyadic-time}, we have
\begin{align}
\forall r\in[0,\tau_n^{\rm step}],\qquad
\frac{n}{2}&\le\bar Z_{\vartheta_n+r}\le2n.
\label{eq:dyadic_strip_bounds_linear}
\end{align}
On \(\mathcal D_n\), the definition of \(\vartheta_n\) gives
\begin{align}
\forall u\in[\theta,\vartheta_n],\qquad
z_0-\|G_1\|_{\infty,[0,\bar T]}&<\bar Z_u\le n.
\label{eq:dyadic_history_Z_bounds}
\end{align}
Since \(s\widetilde Y_1=\bar Z+sG_1\),
the growth bounds on \(F_1\) in Assumption~\myref{I.5},
applied with \(z=s\widetilde Y_{1,u}\),
and boundedness of \(G_1\) give,
for large \(n\),
\[
\inf_{r\in[0,\tau_n^{\rm step}]}
\varepsilon_s F_1(\widetilde Y_{1,\vartheta_n+r})\ge c_1n^{f_{1,s}},
\qquad
\sup_{u\in[\theta,\vartheta_n+\tau_n^{\rm step}]}
|F_1(\widetilde Y_{1,u})|\le C_1n^{f_{1,s}}.
\]
After increasing \(B\) to absorb \(\|\beta\|_{\infty,[0,\bar T]}\)
and \(\sup_{u\in[0,\theta]}|X_u|\), the bound on
\(F_2(\widetilde Y_2)\) in \eqref{eq:Y2_absorption_bound} gives,
for \(t\in[\vartheta_n,\vartheta_n+\tau_n^{\rm step}]\),
\[
\sup_{u\in[0,t]}|\widetilde X_u|\le B+C_1n^{f_{1,s}}+CT_*^{\delta_2f_2}\sup_{u\in[0,t]}|\widetilde X_u|.
\]
We choose \(T_*\in(0,T]\) such that
\[
CT_*^{\delta_2f_2}
\le\min\left(\frac{1}{2},\frac{c_1}{4C_1}\right).
\]
It follows that
\[
\sup_{u\in[0,t]}|\widetilde X_u|\le2B+2C_1n^{f_{1,s}}.
\]
By estimate \eqref{eq:Y2_absorption_bound}, we have
\begin{align*}
|F_2(\widetilde Y_{2,t})|
&\le B+CT_*^{\delta_2f_2}(2B+2C_1n^{f_{1,s}})
\le2B+\frac{c_1}{2}n^{f_{1,s}}.
\end{align*}
For all sufficiently large \(n\), we have
\begin{align*}
\varepsilon_s\widetilde X_t
&\ge\varepsilon_s F_1(\widetilde Y_{1,t})
-|F_2(\widetilde Y_{2,t})|-|\beta(t)|
\ge\frac{c_1}{2}n^{f_{1,s}}-2B-\|\beta\|_{\infty,[0,\bar T]}\ge\frac{c_1}{4}n^{f_{1,s}}.
\end{align*}
For sufficiently large \(n\), the upper bound on \(\sup_{u\in[0,t]}|\widetilde X_u|\) and the lower bound on \(\varepsilon_s\widetilde X_t\) imply

\begin{align}
\inf_{r\in[0,\tau_n^{\rm step}]}
\varepsilon_s\widetilde X_{\vartheta_n+r}
&\ge cn^{f_{1,s}},\label{eq:X_bounds_dyadic_linear}\\
\sup_{u\in[0,\vartheta_n+\tau_n^{\rm step}]}
|\widetilde X_u|
&\le Cn^{f_{1,s}}.\label{eq:X_history_bound_dyadic_linear}
\end{align}
For \(R>Cn^{f_{1,s}}\), we can remove the extra stop at \(\tau_R^{\widetilde X}\).
By the lower bound on \(\varepsilon_s\widetilde X\) in \eqref{eq:X_bounds_dyadic_linear}
and the drift lower bound \eqref{eq:explosive_drift_growth} of Assumption~\myref{V.1}, we have
\[
\forall u\in[\vartheta_n,\vartheta_n+\tau_n^{\rm step}],\qquad
{s\rho}K_1(0)\Gamma_1(u,\widetilde X_u)\sigma(u,\widetilde X_u)
\ge2cn^{f_{1,s}\chi},
\]
for some \(c>0\).
Since \(\bar Z\) is bounded on \([0,\theta]\),
the bounds \eqref{eq:dyadic_history_Z_bounds} and
\eqref{eq:dyadic_strip_bounds_linear} give
\(\sup_{u\in[0,\vartheta_n+\tau_n^{\rm step}]}|\bar Z_u|\le2n\)
for large \(n\).
We set
\[
\forall r\in[0,h_n],\qquad
M_{r}:=sK_1(0)
\int_{\vartheta_n}^{\vartheta_n+(r\wedge\tau_n^{\rm step})}
\Gamma_1(u,\widetilde X_u)\,dW_u.
\]
We write \(\langle M\rangle\) for the predictable quadratic variation
of \(M\).
Subtracting \eqref{eq:resolvent_semimartingale_Zbar} of
Lemma~\ref{lem:stopped_SDE_shifted_Z} at \(\vartheta_n\)
from the same identity at \(\vartheta_n+r\), and using
\(\bar Z_{\vartheta_n}=n\), gives, by
\eqref{eq:memory_pointwise_bound} of the same lemma, with
\(R>Cn^{f_{1,s}}\) from \eqref{eq:X_history_bound_dyadic_linear},
for large \(n\), since \(f_{1,s}\chi>1\),
\begin{align}
\forall r\in[0,\tau_n^{\rm step}],\qquad
\bar Z_{\vartheta_n+r}
&=n+M_{r}+{s\rho}K_1(0)\int_{\vartheta_n}^{\vartheta_n+r}
\Gamma_1(u,\widetilde X_u)\sigma(u,\widetilde X_u)\,du\notag\\
&\quad+\int_{\vartheta_n}^{\vartheta_n+r}
\bigl(r_{\rm mem}(0)\bar Z_u+(dr_{\rm mem}*\bar Z)(u)\bigr)\,du\notag\\
&\ge n+(2cn^{f_{1,s}\chi}-Cn)r+M_{r}\notag\\
&\ge n+cn^{f_{1,s}\chi}r+M_{r}.
\label{eq:dyadic_Z_lower_bound}
\end{align}
By the polynomial growth bound on \(\Gamma_1\) in Assumption~\myref{I.5} and \(\gamma_1f_{1,s}\le\gamma_1f_1\le1\), we have
{
\begin{align}
\langle M\rangle_{h_n}
&=K_1(0)^2
\int_{\vartheta_n}^{\vartheta_n+\tau_n^{\rm step}}
|\Gamma_1(u,\widetilde X_u)|^2\,du\notag\\
&\le C\left(1+
\sup_{u\in[0,\vartheta_n+\tau_n^{\rm step}]}
|\widetilde X_u|^{2\gamma_1}\right)h_n
\le Cn^{2\gamma_1f_{1,s}}h_n\le Cn^2h_n.
\label{eq:quadratic_variation_bound}
\end{align}
}
For sufficiently large \(n\), the identity \(s\widetilde Y_1=\bar Z+sG_1\) gives
\[
\inf_{r\in[0,\tau_n^{\rm step}]}
s\widetilde Y_{1,\vartheta_n+r}
\ge \frac{n}{2}-\|G_1\|_{\infty,[0,\bar T]}>z_0,
\]
so
\(\bar Z_{\vartheta_n+\tau_n^{\rm step}}=2n\) implies
\(\mathcal D_{2n}\) and
\(\vartheta_{2n}=\vartheta_n+\tau_n^{\rm step}\le\vartheta_n+h_n\).
In the definition of \(h_n\) in \eqref{eq:def-dyadic-time}, we choose \(c_0:=\frac{2}{c}\), where \(c>0\) is the coefficient
of \(n^{f_{1,s}\chi}r\) in the lower bound \eqref{eq:dyadic_Z_lower_bound}.
If \(\bar Z_{\vartheta_n+\tau_n^{\rm step}}\ne2n\), the lower bound on \(\bar Z\) in \eqref{eq:dyadic_Z_lower_bound} and
\(cn^{f_{1,s}\chi}h_n=2n\) 
give
\begin{align*}
M_{\tau_n^{\rm step}}
&\le\bar Z_{\vartheta_n+\tau_n^{\rm step}}-n-cn^{f_{1,s}\chi}\tau_n^{\rm step}
\le
\begin{cases}
\frac{n}{2}-n=-\frac{n}{2},
& \bar Z_{\vartheta_n+\tau_n^{\rm step}}=\frac{n}{2},\\
2n-n-cn^{f_{1,s}\chi}h_n=-n,
& \tau_n^{\rm step}=h_n.
\end{cases}
\end{align*}
Doob's inequality applied to \(M^2\) and
the quadratic variation bound \eqref{eq:quadratic_variation_bound} give
\begin{align*}
\PP_0\left(
\inf_{r\in[0,h_n]}M_{r}\le-\frac{n}{2}
\,\middle|\,\mathcal F_{\vartheta_n}
\right)
&\le\frac{4}{n^2}
\E_0[\langle M\rangle_{h_n}\mid\mathcal F_{\vartheta_n}]\le Ch_n,
\end{align*}
which proves \eqref{probdoubling} of Lemma~\ref{lem:dyadic_step_Y_linear}.
\end{proof}

\begin{proof}[Proof of Theorem~\ref{nonmartingality}\myref{nonmartingality_i}]
We use the constants \(T_*,c_0,C\) and the integer \(N_\theta\) from Lemma~\ref{lem:dyadic_step_Y_linear}, and the times \(h_n\) defined in \eqref{eq:def-dyadic-time}.
There exists an a.s.\ finite
\(\mathcal F_\theta\)-measurable integer \(n_0\ge N_\theta\), such that
$n_0>1+\sup_{u\in[0,\theta]}|\bar Z_u|,$
and
\begin{align}
\sum_{k=0}^\infty h_{2^kn_0}
&\le\min\left(\frac{T_*}{2},\frac{1}{2C}\right).
\label{eq:initial_dyadic_sums}
\end{align}
Indeed, \(\bar Z\) is continuous on \([0,\theta]\) and, since \(f_{1,s}\chi>1\), we have
\begin{align*}
\sum_{k=0}^\infty h_{2^kn_0}
&=c_0n_0^{1-f_{1,s}\chi}\sum_{k=0}^\infty(2^{1-f_{1,s}\chi})^k
=\frac{c_0n_0^{1-f_{1,s}\chi}}{1-2^{1-f_{1,s}\chi}}
\xrightarrow[n_0\to\infty]{}0.
\end{align*}
The events
\begin{align*}
\mathcal H_0
&:=\mathcal D_{n_0}\cap
\left\{\vartheta_{n_0}\le\theta+\frac{T_*}{2}\right\},\\
\forall N\in\mathbb N,\qquad \mathcal H_{N+1}
&:=\mathcal H_N\cap\mathcal D_{2^{N+1}n_0}\cap
\{\vartheta_{2^{N+1}n_0}
\le\vartheta_{2^Nn_0}+h_{2^Nn_0}\},
\end{align*}
satisfy \(\mathcal H_N\in\mathcal F_{\vartheta_{2^Nn_0}}\).
The conditional estimates at the random levels
\(2^Nn_0\) follow by partitioning over
\(\{n_0=j\}\in\mathcal F_\theta\), \(j\ge1\).
By Lemma~\ref{lem:reach_from_positive_X_linear}, we have
\[
\PP_0(\mathcal H_0\mid\mathcal F_\theta)>0
\quad\text{a.s.\ on }
\{\varepsilon_s X_\theta>x^*,\ sY_{1,\theta}>z_0\}.
\]
On \(\mathcal H_N\), we have
\[
\vartheta_{2^Nn_0}+h_{2^Nn_0}
\le\theta+\frac{T_*}{2}+\sum_{k=0}^{N}h_{2^kn_0}
\le\theta+T_*.
\]
By Lemma~\ref{lem:dyadic_step_Y_linear} and the tower property,
for \(N\ge0\), we have
\begin{align*}
\PP_0(\mathcal H_{N+1}\mid\mathcal F_\theta)
&=\E_0\left[
\mathbf{1}_{\mathcal H_N}
\PP_0\left(\mathcal H_{N+1}\mid\mathcal F_{\vartheta_{2^Nn_0}}\right)
\,\middle|\,\mathcal F_\theta
\right]\ge(1-Ch_{2^Nn_0})
\PP_0(\mathcal H_N\mid\mathcal F_\theta).
\end{align*}
By induction, for \(N\ge1\), we have
\begin{align*}
\PP_0(\mathcal H_N\mid\mathcal F_\theta)
&\ge{\PP_0}(\mathcal H_0\mid\mathcal F_\theta)
\prod_{k=0}^{N-1}(1-Ch_{2^kn_0})\ge{\PP_0}(\mathcal H_0\mid\mathcal F_\theta)
\left(1-C\sum_{k=0}^{N-1}h_{2^kn_0}\right)\ge\frac{1}{2}\PP_0(\mathcal H_0\mid\mathcal F_\theta).
\end{align*}
Since the events \(\mathcal H_N\) decrease, we have
\[
\PP_0\left(\bigcap_{N\ge0}\mathcal H_N
\,\middle|\,\mathcal F_\theta\right)
\ge\frac{1}{2}\PP_0(\mathcal H_0\mid\mathcal F_\theta)>0
\quad\text{a.s.\ on }
\{\varepsilon_s X_\theta>x^*,\ sY_{1,\theta}>z_0\}.
\]
On \(\bigcap_{N\ge0}\mathcal H_N\),
the lower bound on \(\varepsilon_s\widetilde X\) in \eqref{eq:X_bounds_dyadic_linear} gives
\[
\vartheta_{2^Nn_0}\le\theta+T_*\le\theta+T,
\qquad
\varepsilon_s\widetilde X_{\vartheta_{2^Nn_0}}
\ge c(2^Nn_0)^{f_{1,s}}\xrightarrow[N\to\infty]{}+\infty.
\]
Recalling from Assumption~\myref{standing_ii} that
\[
\tau_{(\theta,n)}^{\widetilde X^\theta}
:=\inf\{u\ge0:\ |\widetilde X_{\theta+u}^{\theta}|\ge n\},
\]
the bound \(|\widetilde X|\ge\varepsilon_s\widetilde X\) gives
\begin{align}
\lim\limits_{n\to\infty}\PP_0\left(
\tau_{(\theta,n)}^{\widetilde X^\theta}\le T
\,\middle|\,\mathcal F_\theta\right)
&\ge{\PP_0}\left(\bigcap_{N\ge0}\mathcal H_N
\,\middle|\,\mathcal F_\theta\right)>0
\label{eq:conditional_explosion}
\end{align}
a.s.\ on
\(\{\varepsilon_s X_\theta>x^*,\ sY_{1,\theta}>z_0\}\).
The bound in \eqref{eq:conditional_explosion} and Lemma~\ref{martingality0} give,
for every bounded stopping time \(\theta\),
\begin{align}
\E_0[S_{\theta+T}\mid\mathcal F_\theta]&<S_\theta
\label{eq:strict_conditional_mass_loss_on_branch}
\end{align}

a.s.\ on
\(\{\varepsilon_s X_\theta>x^*,\ sY_{1,\theta}>z_0\}\).
With \(\theta=t_0\) and \(T=t-t_0\), the supermartingale property gives
\(\E_0[S_t]<\E_0[S_{t_0}]\) 
whenever
\(\PP_0(\varepsilon_s X_{t_0}>x^*,\ sY_{1,t_0}>z_0)>0\).
Thus \(S\) is a strict local martingale on \([t_0,t]\).
\end{proof}

\subsection{Refinement}\label{app:refinement-proof}

The strict conditional inequality in
Theorem~\ref{nonmartingality}\myref{nonmartingality_ii}
concerns the expected asset price after \(\tau_*\), given
\(\mathcal F_{\tau_*}\), on the event \(\{\tau_*<\infty\}\).

\begin{proof}[Proof of Theorem~\ref{nonmartingality}\myref{nonmartingality_ii}]
We fix \(T>0\) and choose any \(N\in\mathbb N^\ast\) such that
\(\PP_0(\tau_*\le N)>0\).
Such an integer exists since
\(\PP_0(\tau_*\le N)\uparrow\PP_0(\tau_*<\infty)>0\)
as \(N\to\infty\).
The event \(\{\tau_*\le N\}\) belongs to \(\mathcal F_{\tau_*}\).
Let \(A\in\mathcal F_{\tau_*}\) satisfy
\(A\subset\{\tau_*\le N\}\). We have
\(A\in\mathcal F_{\tau_*\wedge N}\).
We assume that \(\PP_0(A)>0\).
By the definition of \(\tau_*\) in Assumption~\myref{V.2} and continuity of \(X,Y_1\), we have
\[
A=\bigcup_{h\in\mathbb Q\cap(0,T)}
A\cap\{\varepsilon_s X_{\tau_*+h}>x^*,\ sY_{1,\tau_*+h}>z_0\}.
\]
This union is countable, so one of its events has positive probability.
For the corresponding \(h\),
\eqref{eq:strict_conditional_mass_loss_on_branch} applies at
\((\tau_*\wedge N)+h\), over the interval of length \(T-h\),
on \(A\cap\{\varepsilon_s X_{\tau_*+h}>x^*,\ sY_{1,\tau_*+h}>z_0\}\).
Since \(\tau_*\wedge N=\tau_*\) on \(A\), the tower property and
optional sampling give
\begin{align*}
\E_0[\mathbf1_A S_{\tau_*+T}]
&=\E_0\left[
\mathbf1_A\E_0[S_{(\tau_*\wedge N)+T}
\mid\mathcal F_{(\tau_*\wedge N)+h}]
\right]\\
&<\E_0[\mathbf1_A S_{(\tau_*\wedge N)+h}]\\
&=\E_0\left[
\mathbf1_A\E_0[S_{(\tau_*\wedge N)+h}
\mid\mathcal F_{\tau_*\wedge N}]
\right]\\
&\le\E_0[\mathbf1_A S_{\tau_*}].
\end{align*}
Since \(A\) is arbitrary, the characterization of conditional
expectation gives the claimed strict inequality on \(\{\tau_*\le N\}\).
The union over \(N\) proves the result on \(\{\tau_*<\infty\}\).
\end{proof}

\section{Conclusion}

We have studied the martingale property of stochastic and path-dependent Volterra volatility models in a common setting.  A Girsanov change of measure reduces martingality of \(S\) to non-explosion of \(\widetilde X^\theta\) in \eqref{tildePDV}. On the one hand, we prove martingality either from a one-sided bound on the state, as in Assumption~\myref{III.1},
or from a one-sided growth bound on the Girsanov drift, as in Assumption~\myref{III.2}.
On the other hand, for smooth kernels satisfying Assumption~\myref{IV.2},
the growth and monotonicity conditions in Assumption~\myref{V.1}
and the accessibility condition in Assumption~\myref{V.2}
give strict local martingality.

\medskip

Our results also give the known martingality results for lognormal stochastic volatility and affine Volterra-type models. They give new results for the Volterra--Guyon--Lekeufack model, upside-enhanced Volterra--Guyon--Lekeufack models, the polynomial PDV model, and quadratic Volterra--Heston-type models. They show how the sign of \(\sigma\), the correlation parameter, and the growth of \(\Gamma_1\sigma\) affect whether the price process is a true or strict local martingale.

\medskip

To extend the explosion criterion to rough kernels, we would need a proof that does not use the semimartingale representation of \(\bar Z\).

\clearpage
\bibliographystyle{plainnat}
\bibliography{bibtem}

\clearpage
\appendix

\section{A fractional Sobolev estimate}
\label{app:fractional-sobolev-estimate}

\begin{lemma}\label{lem:fractional_sobolev_sup}
For \(r>2\), there exists
\(C_r>0\) such that every real-valued process \(\varphi\) with
continuous paths on \([0,T]\) satisfies, for \(t\in(0,T]\),
\begin{align}
\E_0\left[\sup_{u\in[0,t]}|\varphi_u-\varphi_0|^r\right]
&\le C_rt
\int_0^t\int_0^t
\frac{\E_0[|\varphi_v-\varphi_u|^r]}{|v-u|^3}\,du\,dv.
\label{eq:fractional_sobolev_sup}
\end{align}
\end{lemma}
This is a consequence of Corollary~A.2 of \cite{FrizVictoir2010}, with \(q=r\) and
\(\alpha=\frac{2}{r}\).

\section{Proof of Lemma~\ref{lem:stopped_SDE_shifted_Z}}\label{app:proof-stopped-SDE-shifted-Z}

Under the kernel regularity in Assumption~\myref{IV.2},
we apply \cite[Chapter~2, Theorem~3.5]{GripenbergLondenStaffans1990} with
\[
\forall t\in[0,+\infty),\qquad
k(t)=\frac{K_1'(t)}{K_1(0)},
\qquad
f(t)=-\frac{K_1'(t)}{K_1(0)^2}.
\]
This gives a unique solution
\(\ell_{\rm res}\in C(\mathbb R_+)\cap BV_{\rm loc}(\mathbb R_+)\)
to the Volterra equation below.
\begin{align}
\ell_{\rm res}+\frac{K_1'}{K_1(0)}*\ell_{\rm res}
&=-\frac{K_1'}{K_1(0)^2}.
\end{align}
We define \(\mathcal R_1\) by
\begin{align}
\mathcal R_1:=\frac{\delta_0}{K_1(0)}+\ell_{\rm res}(t)\,dt.
\label{eq:smooth-canonical-resolvent}
\end{align}
Since \((K_1*\mathcal R_1)(0)=1\) and
\[
\forall t>0,\qquad
(K_1*\mathcal R_1)'(t)
=\frac{K_1'(t)}{K_1(0)}+K_1(0)\ell_{\rm res}(t)
+(K_1'*\ell_{\rm res})(t)=0,
\]
we have \(K_1*\mathcal R_1=1\).
The function \(r_{\rm mem}\) defined in \eqref{eq:smooth-memory-kernel} therefore satisfies
\[
r_{\rm mem}=K_1'*\mathcal R_1
=\frac{K_1'}{K_1(0)}+K_1'*\ell_{\rm res}
=-K_1(0)\ell_{\rm res}
\in C(\mathbb R_+)\cap BV_{\rm loc}(\mathbb R_+).
\]

\begin{proof}
For Lemma~\ref{lem:stopped_SDE_shifted_Z}\myref{stopped_SDE_shifted_Z_i}, we fix \(R\in\mathbb N^\ast\) and set
\[
\forall t\in[0,+\infty),\qquad
U_t^{(R)}
:=
s
\int_0^t
\mathbf{1}_{\{r\le\tau_R^{\widetilde X}\}}
\Gamma_1(r,\widetilde X_{r})\,dW_{r}
+
s\rho\int_0^t
\mathbf{1}_{\{r\le\tau_R^{\widetilde X}\}}
\mathbf{1}_{\{r\ge\theta\}}
\Gamma_1(r,\widetilde X_{r})\sigma(r,\widetilde X_{r})\,dr.
\]
For all $t \in [0,\tau_R^{\widetilde X}]$, we have
$\bar Z_t=(K_1\ast dU^{(R)})(t).$
On each bounded time interval, stopping at \(\tau_R^{\widetilde X}\) makes the coefficients bounded. The derivative \(K_1'\) is also bounded on this interval. Writing
\[
\forall t\ge0,\qquad K_1(t)=K_1(0)+\int_0^tK_1'(u)\,du,
\]
the stochastic Fubini theorem
\cite[Theorem~2.2]{Veraar2012} gives
\begin{align}\label{eq:resolvent_semimartingale_Z_case}
\forall t \in [0,\tau_R^{\widetilde X}],
\qquad
\bar Z_t
&=
K_1(0)U_t^{(R)}
+
\int_0^t (K_1'\ast dU^{(R)})(r)\,dr.
\end{align}
Since \(\mathcal R_1*K_1=1\), we have
\(r_{\rm mem}*K_1=K_1'*(\mathcal R_1*K_1)=K_1-K_1(0)\).
Differentiating gives
\[
K_1'=r_{\rm mem}(0)K_1+dr_{\rm mem}*K_1.
\]
Applying \cite[Lemma~2.1]{AbiJaberLarssonPulido2019}
with kernel \(K_1\) and measure \(dr_{\rm mem}\)
to the martingale part \(M^{(R)}\) of \(U^{(R)}\),
whose quadratic variation has a locally bounded density, gives
\[
\forall t\ge0,\qquad
\bigl(dr_{\rm mem}*(K_1*dM^{(R)})\bigr)(t)
=\bigl((dr_{\rm mem}*K_1)*dM^{(R)}\bigr)(t).
\]
Together with ordinary Fubini for the drift part, this gives up to indistinguishability, for
\(t\in[0,\tau_R^{\widetilde X}]\),
\begin{align*}
(K_1'*dU^{(R)})(t)
&=r_{\rm mem}(0)(K_1*dU^{(R)})(t)
+\bigl((dr_{\rm mem}*K_1)*dU^{(R)}\bigr)(t)\\
&=r_{\rm mem}(0)(K_1*dU^{(R)})(t)
+\bigl(dr_{\rm mem}*(K_1*dU^{(R)})\bigr)(t)\\
&=r_{\rm mem}(0)\bar Z_t+(dr_{\rm mem}*\bar Z)(t),
\end{align*}
since \(K_1*dU^{(R)}=\bar Z\) on \([0,\tau_R^{\widetilde X}]\).
Together with \eqref{eq:resolvent_semimartingale_Z_case},
this proves Lemma~\ref{lem:stopped_SDE_shifted_Z}\myref{stopped_SDE_shifted_Z_i}.
The bound in Lemma~\ref{lem:stopped_SDE_shifted_Z}\myref{stopped_SDE_shifted_Z_ii} follows from
\[
\forall t\in[0,\bar T\wedge\tau_R^{\widetilde X}],\qquad
|(dr_{\rm mem}*\bar Z)(t)|
\le\int_{(0,t]}|\bar Z_{t-r}|\,|dr_{\rm mem}|(dr)
\le |dr_{\rm mem}|([0,\bar T])\sup_{u\in[0,t]}|\bar Z_u|.
\]
\end{proof}

\section{\texorpdfstring{Proof of Lemma~\ref*{lem:accessibility_models}}{Proof of the accessibility lemma}}
\label{app:accessibility_models}

We complete the proof of Lemma~\ref{lem:accessibility_models}
for the polynomial PDV model, specification~\eqref{eq:F1A},
and the Quadratic Volterra Heston+ model.

\subsection{Polynomial PDV and specification (A)}
\label{app:accessibility_pdv_A}

We use the assumptions of Theorem~\ref{nonmartingality}.
For the polynomial PDV model and the model with specification~\eqref{eq:F1A}, respectively, we have
\[
 X={\beta+F_1(Y_1)}+F_2(Y_2),\qquad
 \sigma=\sqrt X\ \text{or}\ X,\qquad
 Y_2=G_2+K_2*\sigma^2\ge G_2{\ge0}.
\]
Here \(\beta\) is constant in both models, and
we have $z_0=-\frac{\beta_1}{2\beta_1^{(2)}}$ or $\frac{\beta_{0,1}-\beta_{0,2}}{\beta_{1,2}-\beta_{1,1}}$.
In both models, we have \(\Gamma_1=\sigma\) and \(\Gamma_2=\sigma^2\).
The function \(F_1\) decreases on \((-\infty,z_0]\) and has its global minimum at \(z_0\). Also, we have \(\lim\limits_{y\to-\infty}F_1(y)=+\infty\), and \(F_2\) is nondecreasing and unbounded above.
We assume that \(X_0>0\) and fix \(T>0\).
If \(Y_{1,0}>z_0\), then \(X_0>0\) gives \(\tau_*=0<T\).
We now consider \(Y_{1,0}\le z_0\).
\begin{rk}
Moving directly towards \(z_0\) may bring \(X\) to zero.
We therefore choose a deterministic reference path for \(Y_1\)
that first takes sufficiently negative values and then reaches a value above \(z_0\).
Since \(X={\beta+F_1(Y_1)}+F_2(Y_2)\), large negative values of \(Y_1\) make \({\beta+F_1(Y_1)}\), \(X\), and \(\sigma^2\) large.
This creates a large contribution to \(Y_2\) through
\(K_2*\sigma^2\). If the reference path moves from these negative values to a value above \(z_0\) in a sufficiently short time, as quantified in \eqref{eq:accessibility_K2_window_lower_bound},
this contribution keeps \(F_2(Y_2)\) large enough so that
$X={\beta+F_1(Y_1)}+F_2(Y_2)\ge {\beta+F_1(z_0)}+F_2(Y_2)>0.$
Thus the path crosses \(z_0\) while the diffusion coefficient stays nonzero, with \(\Gamma_1=\sigma\ne0\).
\end{rk}
For the polynomial PDV model, the quadratic expression for \(X\) in
\eqref{eq:poly-pdv} gives
\begin{align*}
\forall t\in[0,T],\qquad
\sigma_t^2&=X_t\le C(1+Y_{1,t}^2)+C|Y_{2,t}|.
\end{align*}
For specification~\eqref{eq:F1A}, the linear growth of \(F_1\) gives
\begin{align*}
\forall t\in[0,T],\qquad
\sigma_t^2
&=\bigl({\beta+F_1(Y_{1,t})}+\beta_2\sqrt{Y_{2,t}}\bigr)^2\\
&\le{2(\beta+F_1(Y_{1,t}))^2}+2\beta_2^2Y_{2,t}
\le C(1+Y_{1,t}^2)+C|Y_{2,t}|.
\end{align*}
In both models, \(Y_2=G_2+K_2*\sigma^2\) and \(K_2\ge0\), so we have
\[
\forall t\in[0,T],\qquad
|Y_{2,t}|\le\|G_2\|_{\infty,[0,T]}
+\int_0^tK_2(t-s)\sigma_s^2\,ds.
\]
Substituting the bound on \(|Y_{2,t}|\) into the corresponding bound on \(\sigma_t^2\) gives, for some \(C>0\),
\begin{align}
\forall t\in[0,T],\qquad
\sigma_t^2&\le C(1+Y_{1,t}^2)
+C\int_0^tK_2(t-s)\sigma_s^2\,ds.
\label{eq:accessibility_sigma_growth}
\end{align}

\medskip
\noindent\textbf{Step 1: Moving to the left with positive volatility.} The map
\((s,a)\mapsto\beta+F_1((G_1(s)+a)\wedge z_0)+F_2(G_2(s))\)
is continuous at \((0,0)\), where its value is \(X_0>0\).
Hence there exist \(b\in(0,\frac{T}{2})\) and \(\delta\in(0,\frac{1}{2}]\) such that
\begin{align}
\forall s\in[0,b],\qquad
{\beta+F_1((G_1(s)+\delta)\wedge z_0)}+F_2(G_2(s))
&\ge\frac{X_0}{2}.
\label{eq:accessibility_initial_lower_bound}
\end{align}
Continuity of \(v\mapsto\int_v^{v+\frac{b}{2}}K_2(u)\,du\) and
the inequality
\(\int_0^{\frac{b}{2}}K_2(u)\,du>0\) from Assumption~\ref{assumption1}
give \(h\in(0,T-b)\) such that
\begin{align}
\inf_{v\in[0,h]}\int_v^{v+\frac{b}{2}}K_2(u)\,du
&\ge\frac{1}{2}\int_0^{\frac{b}{2}}K_2(u)\,du>0.
\label{eq:accessibility_K2_window_lower_bound}
\end{align}
For \(R\ge0\), the function
\[
\forall s\in[0,T],\qquad
p(s)=-Rs+\frac{((s-b)^+)^2}{h^2}
(z_0+1-G_1(b+h)+R(b+h))
\]
belongs to \(C^1([0,T])\), equals \(-Rs\) on \([0,b]\), and satisfies
\(G_1(b+h)+p(b+h)=z_0+1\). We define
\[
\tau:=(b+h)\wedge
\inf\{s\in[0,b+h]:
|Y_{1,s}-G_1(s)-p(s)|\ge\delta\}.
\]
Since \(p(0)=0\) and \(Y_1-G_1-p\) is continuous, we have
\[
\sup_{s\in[0,\tau]}|Y_{1,s}-G_1(s)-p(s)|\le\delta.
\]
In particular, we have
\begin{align}
\forall s\in[0,b\wedge\tau],\qquad
Y_{1,s}&\le G_1(s)-Rs+\delta\le G_1(s)+\delta.
\label{eq:accessibility_Y1_upper_bound}
\end{align}
Since \(F_1\) is nonincreasing on \((-\infty,z_0]\) and has its global minimum at \(z_0\), we have
\[
\forall s\in[0,b\wedge\tau],\qquad
F_1(Y_{1,s})\ge {F_1((G_1(s)+\delta)\wedge z_0)}.
\]
Also, \(Y_2\ge G_2\) and \(F_2\) is nondecreasing. Hence,
by the initial lower bound \eqref{eq:accessibility_initial_lower_bound}, we have
\[
\forall s\in[0,b\wedge\tau],\qquad
X_s\ge {\beta+F_1((G_1(s)+\delta)\wedge z_0)}+F_2(G_2(s))
\ge\frac{X_0}{2}.
\]
On \(\{\tau<b\}\), we have \(X_s\ge \frac{X_0}{2}\) for every \(s\in[0,\tau]\).
On \(\{\tau\ge b\}\), for all sufficiently large \(R\),
the bound on \(Y_1\) in \eqref{eq:accessibility_Y1_upper_bound} gives
\begin{align*}
\sup_{s\in[\frac{b}{2},b]}Y_{1,s}
&\le{\|G_1\|_{\infty,[0,b]}-\frac{Rb}{2}+\delta}<z_0,\\
\inf_{s\in[\frac{b}{2},b]}X_s
&\ge {\beta+F_1\left(\|G_1\|_{\infty,[0,b]}-\frac{Rb}{2}+\delta\right)}
+\min_{s\in[0,b]}F_2(G_2(s))
\xrightarrow[R\to\infty]{}+\infty.
\end{align*}
For sufficiently large \(R\), this lower bound is at least \(1\).
Thus we have \(\sigma_s^2\ge X_s\) on \([\frac{b}{2},b]\) in both models,
and we have
\begin{align}
\inf_{s\in[\frac{b}{2},b]}\sigma_s^2
&\ge\beta+F_1\left(\|G_1\|_{\infty,[0,b]}-\frac{Rb}{2}+\delta\right)
+\min_{s\in[0,b]}F_2(G_2(s))
\xrightarrow[R\to\infty]{}+\infty.
\label{eq:accessibility_sigma_lower_divergence}
\end{align}

\medskip
\noindent\textbf{Step 2: Crossing \(z_0\) while keeping \(X\) positive.} On \(\{\tau\ge b\}\), for every \(v\in[0,h]\),
the lower bound on the integral of \(K_2\) in \eqref{eq:accessibility_K2_window_lower_bound} gives
{
\[
\int_{\frac{b}{2}}^b K_2(b+v-s)\,ds
=\int_v^{v+\frac{b}{2}}K_2(u)\,du
\ge\frac{1}{2}\int_0^{\frac{b}{2}}K_2(u)\,du.
\]
Consequently, we have
}
\begin{align*}
Y_{2,b+v}
&\ge G_2(b+v)+\int_{\frac{b}{2}}^bK_2(b+v-s)\sigma_s^2\,ds
\ge\min_{s\in[0,T]}G_2(s)
+\frac{1}{2}\left(\inf_{s\in[\frac{b}{2},b]}\sigma_s^2\right)
\int_0^{\frac{b}{2}}K_2(u)\,du.
\end{align*}
By \eqref{eq:accessibility_sigma_lower_divergence}, since \(F_1\ge F_1(z_0)\) and \(F_2(y)\xrightarrow[y\to+\infty]{}+\infty\),
there exists \(R\) large enough that, on \(\{\tau\ge b\}\),
\begin{align}
\forall v\in[0,h],\qquad
X_{b+v}&\ge {\beta+F_1(z_0)}+F_2(Y_{2,b+v})\ge1.
\label{eq:accessibility_positive_return}
\end{align}
We keep this choice of \(R\), \(p\), and \(\tau\). Then we have
\[
\forall s\in[0,\tau],\qquad
X_s\ge\frac{X_0}{2}\wedge1.
\]

\medskip
\noindent\textbf{Step 3: Applying the tube estimate.} The definition of \(\tau\) gives
\[
\sup_{s\in[0,\tau]}|Y_{1,s}|
\le\|G_1\|_{\infty,[0,T]}+\|p\|_{\infty,[0,T]}+\delta.
\]
The Volterra--Gr\"onwall lemma \cite[Lemma~2.2]{Zhang2010} applied to
the bound on \(\sigma^2\) in \eqref{eq:accessibility_sigma_growth} on \([0,\tau]\) gives
a finite constant \(C\), depending on \(p,\delta,T\) and the
initial curves, such that
\begin{align}
\sup_{s\in[0,\tau]}|\sigma_s|&\le C.
\label{eq:accessibility_sigma_upper}
\end{align}
Therefore, for some \(0<c\le C<\infty\), we have
\[
\forall s\in[0,\tau],\qquad
c\le|\Gamma_1(s,X_s)|=|\sigma_s|\le C.
\]
As in the proof of Lemma~\ref{lem:stopped_SDE_shifted_Z}, the model
\eqref{PDVeq} gives the semimartingale representation
\begin{align}
\forall t\in[0,\tau],\qquad
Y_{1,t}-G_1(t)
&=K_1(0)\int_0^t\Gamma_1(u,X_u)\,dW_u
\notag\\
&\quad+\int_0^t\left[
r_{\rm mem}(0)(Y_{1,u}-G_1(u))
+\bigl(dr_{\rm mem}*(Y_1-G_1)\bigr)(u)
\right]du.
\label{eq:accessibility_pdv_semimartingale}
\end{align}
The bound \eqref{eq:memory_pointwise_bound}, with \(Y_1-G_1\) in place of
\(\bar Z\), bounds its drift on \([0,\tau]\).
 Lemma~\ref{lem:accessibility_tube}, applied to
\(Y_1-G_1\) on \([0,b+h]\), with stopping time \(\tau\), gives
\[
\PP_0\left(
\sup_{s\in[0,\tau]}|Y_{1,s}-G_1(s)-p(s)|<\delta
\right)>0.
\]
On \(\{\tau<b+h\}\), continuity of \(Y_1-G_1-p\) gives
\(|Y_{1,\tau}-G_1(\tau)-p(\tau)|=\delta\).
Therefore we have \(\PP_0(\tau=b+h)>0\).
On \(\{\tau=b+h\}\), \eqref{eq:accessibility_positive_return} gives \(X_{b+h}\ge1\),
and the definition of \(\tau\) gives
\begin{align*}
Y_{1,b+h}
&\ge G_1(b+h)+p(b+h)-\delta
=z_0+1-\delta
\ge z_0+\frac{1}{2}.
\end{align*}
Hence we have
\begin{align}
\PP_0(X_{b+h}>0,\ Y_{1,b+h}>z_0)
&\ge\PP_0(\tau=b+h)>0.
\label{eq:accessibility_crossing_probability}
\end{align}
Since \(s=\varepsilon_s=1\), \(x^*=0\), and \(b+h<T\),
\eqref{eq:accessibility_crossing_probability} gives
\[
\PP_0(\tau_*<T)
\ge\PP_0(X_{b+h}>0,\ Y_{1,b+h}>z_0)>0,
\]
which proves \eqref{eq:accessibility_models} for these two models.

\subsection{Quadratic Volterra Heston+}
\label{app:accessibility_qvh}

For the Quadratic Volterra Heston+ model with \(a_0>0\) and \(1\le p\le2\),
under the assumptions of Corollary~\ref{corollary}\myref{corollary_ii},
we have \(X=Y_1\), so \(X-G_1\) has the semimartingale representation
\eqref{eq:accessibility_pdv_semimartingale}.
For \(t>0\), \(a\in\mathbb R\), and \(\delta>0\), we define the path \(p\) by
\[
\forall s\in[0,t],\qquad p(s):=\frac{s}{t}(a-G_1(t)),
\]
and consider the tube
\begin{align}
&\left\{(u,x)\in[0,t]\times\mathbb R:
|x-G_1(u)-p(u)|<\delta\right\}.
\label{eq:accessibility_qvh_tube}
\end{align}
We set
\[
\tau:=t\wedge\inf\{u\in[0,t]:
|X_u-G_1(u)-p(u)|\ge\delta\}.
\]
On \([0,\tau]\), \(\Gamma_1=\sqrt{\widehat V}\)
is bounded and bounded below by \(\sqrt{a_0}>0\).
On \([0,\tau]\), the drift of \(X-G_1\) is also bounded by
\eqref{eq:memory_pointwise_bound} of
Lemma~\ref{lem:stopped_SDE_shifted_Z}.
If \(\sup_{u\in[0,\tau]}|X_u-G_1(u)-p(u)|<\delta\),
continuity gives \(\tau=t\), and \(|X_t-a|<\delta\) follows from
\(G_1(t)+p(t)=a\).
Lemma~\ref{lem:accessibility_tube}, applied to the tube
\eqref{eq:accessibility_qvh_tube} up to \(\tau\), therefore gives
\begin{gather*}
\forall t>0,\quad\forall a\in\mathbb R,\quad\forall\delta>0,  \quad\PP_0(|X_t-a|<\delta)
\ge\PP_0\!\left(
\sup_{u\in[0,\tau]}|X_u-G_1(u)-p(u)|<\delta
\right)>0.
\end{gather*}
In particular, we have \(\PP_0(X_t>0)>0\) and \(\PP_0(X_t<0)>0\) for every
\(t>0\).
Here \(X_t=Y_{1,t}\) and \(\varepsilon_s=s\), so the full-support
estimate also gives
\[
\forall t>0,\qquad
\PP_0(\varepsilon_s X_t>x^*,\ sY_{1,t}>z_0)
=\PP_0(sX_t>\max\{x^*,z_0\})>0.
\]
Taking \(t=T/2\) gives \(\PP_0(\tau_*<T)>0\), completing
the proof of Lemma~\ref{lem:accessibility_models}.

\end{document}